\documentclass{amsart}
\date{November 6, 2010}
\usepackage[dvips]{graphicx} 
\usepackage{verbatim,enumerate}
\usepackage[usenames]{color}
\title[Coherent tangent bundles and Gauss-Bonnet formulas]{%
    Coherent tangent bundles\\
    and Gauss-Bonnet formulas for wave fronts%
}
\author{Kentaro Saji}
\address[Saji]{%
  Department of Mathematics,
  Faculty of Education,
  Gifu University,  Yanagido 1-1, Gifu 501-1193, Japan
}
\email{ksaji@gifu-u.ac.jp}

\author{Masaaki Umehara}
\address[Umehara]{%
   Department of Mathematics, Graduate School of Science,
   Osaka University,
   Toyonaka, Osaka 560-0043,
   Japan
}
\email{umehara@math.sci.osaka-u.ac.jp}
\author{Kotaro Yamada}
\address[Yamada]{%
   Department of Mathematics,
   Tokyo Institute of Technology,
   O-okayama, Meguro, Tokyo 152-8550, Japan%
}
\email{kotaro@math.titech.ac.jp}
\dedicatory{%
  Dedicated to Professor Toshiki Mabuchi on 
  the occasion of his sixtieth birthday
}

\thanks{%
 K. Saji, M. Umehara and K. Yamada were partially 
 supported by Grant-in-Aid for Scientific Research 
 (Young Scientists (B)) No. 20740028,
 (A) No.22244006
 and (B) No. 21340016,
 respectively from the Japan Society for the Promotion of Science.
}

\renewcommand{\keywords}[1]{\ignorespaces\relax\ignorespaces}
\newcommand{\subclass}[1]{\ignorespaces\relax\ignorespaces}
\usepackage{amsthm}
\theoremstyle{plain}
 \newtheorem{theorem}{Theorem}[section]
 \newtheorem{proposition}[theorem]{Proposition}
 \newtheorem{fact}[theorem]{Fact}
 \newtheorem{lemma}[theorem]{Lemma}
 \newtheorem{corollary}[theorem]{Corollary}
\theoremstyle{definition}
 \newtheorem{definition}[theorem]{Definition}
\theoremstyle{remark}
 \newtheorem{remark}[theorem]{Remark}
 \newtheorem*{remark*}{Remark}
 \newtheorem{example}[theorem]{Example}
\numberwithin{equation}{section}
\renewcommand{\theenumi}{{\rm(\arabic{enumi})}}
\renewcommand{\labelenumi}{\theenumi}
\newenvironment{enum}{%
  \begin{enumerate}\setlength{\itemindent}{1em}%
                   \setlength{\leftmargin}{20em}%
}{\end{enumerate}}
\newcommand{\vect}[1]{\boldsymbol{#1}}
\newcommand{\R}{\boldsymbol{R}}

\renewcommand{\phi}{\varphi}
\newcommand{\sgn}{\operatorname{sgn}}
\newcommand{\ext}{\operatorname{ext}}
\newcommand{\inner}[2]{\left\langle{#1},{#2}\right\rangle}
\newcommand{\Ker}{\operatorname{Ker}}

\newcommand{\SO}{\operatorname{SO}}

\newcommand{\E}{\mathcal{E}}

\newcommand{\op}[1]{\operatorname{#1}}

\newcommand{\trans}[1]{\vphantom{#1}^t\!#1}
\newcommand{\first}{\operatorname{\mathit{I}}}
\newcommand{\second}{\operatorname{\mathit{I\!I}}}
\newcommand{\third}{\operatorname{\mathit{I\!I\!I}}}

\begin{document}
\begin{abstract}
 We give a definition of `coherent tangent bundles',
 which is an intrinsic formulation of wave fronts. 
 In our application of coherent tangent bundles
 for wave fronts, the first fundamental forms and 
 the third fundamental forms 
 are considered as induced metrics of 
 certain homomorphisms between vector bundles.
 They satisfy the completely same conditions, 
 and so can reverse roles with  each other.
 For a given wave front of a $2$-manifold, 
 there are two  Gauss-Bonnet formulas.
 By exchanging the roles of the fundamental forms,
 we get two new additional Gauss-Bonnet formulas
 for the third fundamental form. 
 Surprisingly, these are different from those for 
 the first fundamental form, and using these four formulas, 
 we get several new results on the topology 
 and geometry of wave fronts.
\end{abstract}
\maketitle

\section{Introduction}
In this paper, we give a definition of
{\it coherent tangent bundles}, which is an intrinsic 
formulation of wave fronts. 
The advantage of this formulation is that the first fundamental
forms and the third fundamental forms satisfy 
exactly the same conditions for wave fronts in space forms,
and they can reverse roles with each other.
Then the two Gauss-Bonnet formulas induce two more Gauss-Bonnet
formulas by exchanging the first and third fundamental forms.
These turn out to be different from those for the first fundamental
form, and using these four formulas, we get new results on fronts.

In Section~\ref{sec:coherent}, 
we generalize the definitions of {\em coherent tangent bundles\/} and 
{\it  singular curvature}, which were given for $2$-dimensional
coherent tangent bundles (or fronts)
in \cite{SUY1} and \cite{SUY2}, to fronts of general dimension.
A coherent tangent bundle on an $m$-manifold $M^m$ induces
a positive semi-definite metric on $M^m$
(called the first fundamental form), and
can be regarded as a generalization of Riemannian manifolds.
Using  the concept of coherent tangent bundles,
we can give a unified treatment of
hypersurfaces and $C^\infty$-maps between
the same dimensional manifolds.
The points where the metric degenerates are called
{\it singular points}. 
We define a notion of $A_k$ (singular) points, 
which is a generalization of that 
of fronts and of $C^\infty$-Morin maps at the same time.
We define {\it singular principal curvatures\/}
for each $A_2$-point
of a given coherent tangent bundle, which 
are $(m-1)$-tuples of real numbers.
One of them diverges to $-\infty$ at
$A_3$-points as shown in Theorem \ref{thm:infty}.
This is a generalization of the result in \cite{SUY1}
for $2$-dimensional fronts.
Moreover, 
as shown in Section~\ref{sec:gauss-bonnet},
our intrinsic setting enables us
to introduce the singular  
curvature on $A_2$-Morin singular
points.
When $m=2$, as an application of the two intrinsic Gauss-Bonnet
formulas, we give a new proof of 
Quine's formula (cf.\ Proposition~\ref{prop:quine}), 
and also get a new formula
(cf.\ Proposition~\ref{prop:id})
for the total singular curvature of generic smooth maps between
$2$-manifolds.

Furthermore, in Section~\ref{sec:gauss-bonnet},
we also give  several results on compact $2$-dimensional fronts,
as an application of our new Gauss-Bonnet formulas
for the third fundamental form.
In particular,  we show that the total negative Gaussian 
curvature $\int_{M^2}\min(0,K)\,dA$ of a closed 
immersed surface $M^2$ in $\R^3$ is equal to the 
signed sum of total geodesic curvature 
on each connected component of the singular set of
its Gauss map, see Theorem~\ref{thm:c}.
The deepest applications are given for
surfaces of bounded Gaussian curvature.
For example, the Euler characteristic of
a closed wave front with Gaussian curvature
$-1/\delta<K<-\delta$ ($\delta>0$)
in a $3$-dimensional flat torus 
vanishes (cf.\ Theorem~\ref{thm:g}).

\section{Coherent tangent bundles}
\label{sec:coherent}
Let $M^m$ be an oriented $m$-manifold ($m\ge 1$)
and  $\E$ a vector bundle of rank $m$ over $M^m$.
Before defining coherent tangent bundles,
we define $A_k$-points
for each vector bundle homomorphism 
from the tangent bundle of $M^m $to $\E$.

\subsection{$A_k$-points}
We assume that $\E$ is orientable, namely there exists
a non-vanishing  section $\mu:M^m\to \op{det}(\E^*)$,
where $\det(\E^*)$ is the determinant line bundle of 
the dual bundle $\E^*$.
This assumption does not create any restriction for
the local theory. In fact, for each point $p\in M^m$,
there exists  a section $\mu$ of $\det(\E^*)\setminus\{0\}$
defined on  a neighborhood $U$ of $p$. 
Now we fix a bundle homomorphism
\[
   \phi:TM^m\longrightarrow \E,
\]
where $TM^m$ is the tangent bundle of $M^m$.
A point $p\in M^m$ is called a {\it $\phi$-singular point\/} 
if $\phi_p\colon{}T_pM^m\to\E_p$ is not a bijection, where $\E_p$ is the
fiber of $\E$ at $p$.
We denote by $\Sigma_{\phi}$ the set of $\phi$-singular points on $M^m$.
On the other hand, a point $p\in M^m\setminus\Sigma_{\phi}$
is called a {\it $\phi$-regular point}.
We now fix a $\phi$-singular point $p$, and take a positively oriented
local coordinate neighborhood $(U;u_1,\dots,u_m)$ of $p$.
THen it holds that
\[
     \varphi^*\mu=\lambda_\varphi\,du_1\wedge\dots\wedge du_m,
\]
where 
\begin{equation}\label{eq:jacobian}
  \lambda_\varphi
       =\mu(\varphi_1,\ldots,\varphi_n):U\to\R%
       ,\qquad
      \left(
         \varphi_j=
	     \phi\left(\frac{\partial}{\partial u_j}\right)%
	     ,\ j=1,\dots,m
      \right)
\end{equation}
is a $C^\infty$-function.
We call the function $\lambda_\varphi$ the
{\em $\varphi$-Jacobian function}.
The set of $\phi$-singular points on $U$ is expressed as
\begin{equation}\label{eq:singular0}
   \Sigma_{\phi}\cap U:=\{p\in U\,;\,\lambda_{\phi}(p)=0  \}. 
\end{equation}
We now set
\begin{equation}\label{eq:M_pm}
   M^+_\phi:=\{q\in M^m\,;\, \lambda_\phi(q)>0\},\qquad
   M^-_\phi:=\{q\in M^m\,;\, \lambda_\phi(q)<0\}.
\end{equation}
Since the sign of $\lambda_\phi$ does not depend on
the choice of the positively oriented local coordinate system,
the above definitions of $M^\pm_\phi$ make sense.

A $\phi$-singular point $p$ $(\in\Sigma_{\phi})$ is called 
{\em non-degenerate\/} if $d\lambda_\phi$ does not vanish at $p$.
On a neighborhood of a non-degenerate $\phi$-singular point,
the $\phi$-singular set consists of an $(m-1)$-submanifold in $M^m$, 
called the {\em $\phi$-singular submanifold}.
If $p$ is a non-degenerate $\phi$-singular point, the rank of 
$\phi_p$ is $m-1$.
The direction of the kernel of $\phi_p$ is 
called the {\em  null direction}.
Let $\eta$ be {\it a null vector field}, 
namely, smooth (non-vanishing) 
vector field along the
$\phi$-singular submanifold $\Sigma_\phi$, 
which gives the null direction
at each point in $\Sigma_\phi$.

\begin{definition}[$A_2$-points]%
\label{def:a2-point}
 A non-degenerate $\phi$-singular point $p\in M^m$
 is called an 
 {\em $A_2$-point\/}, or
 an {\em $A_2$-point of $\phi$\/}, 
 if the null direction $\eta(p)$ is transversal to the 
 $\varphi$-singular
 submanifold.
\end{definition}

We set
\begin{equation}\label{eq:lambda-prime}
       \lambda_\phi':=d\lambda(\tilde\eta),
\end{equation}
where $\tilde\eta$ is a vector field on a neighborhood $U$ of $p$
which gives a null vector field  on $\Sigma_\phi\cap U$.
Then $p$ is an $A_2$-point if and only if the function
$\lambda'_\phi$ does not vanish at $p$ (see \cite[Theorem 2.4]{SUY4}).
\begin{definition}[$A_3$-points]
\label{def:a3-point}
 A non-degenerate $\varphi$-singular point $p\in M^m$ 
 is called an {\em $A_3$-point} or an {\em $A_3$-point of $\varphi$}
 if it is not an $A_2$-point of $\varphi$, 
 but $d\lambda_\varphi'(\eta)$ does not
vanish.
\end{definition}

Similarly, we can define $A_{k+1}$-points inductively.
We set
\begin{equation}\label{eq:lambda-k}
 \lambda^{(k)}_\varphi
  :=
  d\lambda^{(k-1)}_\varphi(\tilde{\eta})
  :
  \Sigma_{\varphi}\to{\R}\qquad
    (\lambda_\varphi^{(1)}:=\lambda_\varphi',\ 
     \lambda_\varphi^{(0)}:=\lambda_\varphi).
\end{equation}
If $(d\lambda^{(k-1)}_\varphi)_p\ne0$
on $T_p\Sigma^{k-1}$,
then we can define a subset of
the submanifold $\Sigma_\varphi^k$ by
\[
   \Sigma^{k+1}_\varphi
      :=
      \{q\in\Sigma_\varphi^k;\lambda^{(k)}_\varphi(q)=0\}
      =
      \{q\in\Sigma_\varphi^k;\eta(q)\in T_q\Sigma_\varphi^k\},
\]
where we set $\Sigma_\varphi^1:=\Sigma_\varphi$
and $\Sigma_\varphi^0:=M^m$.
If a non-degenerate $\varphi$-singular point
$p$ is not an $A_k$-point of $\varphi$,
but $d\lambda^{(k-1)}_\varphi(\tilde\eta)$
is well-defined inductively and is nonvanishing
at $p$,
then it is called an {\em $A_{k+1}$-point of} $\varphi$.
Here an $A_1$-point means a $\varphi$-regular point.
(See [17, Section 2] for details.)

\subsection{Coherent tangent bundles}
As a generalization of the $2$-dimensional case in \cite{SUY1} and
\cite{SUY2}, 
we give a general setting for intrinsic fronts:
Let $M^m$ be an oriented $m$-manifold ($m\ge 1$).
A {\em coherent tangent bundle\/} over $M^m$ 
is a $5$-tuple $(M^m,\E,\inner{~}{~},D,\varphi)$, 
where 
\begin{enum}
 \item $\E$ is a vector bundle of rank $m$ over $M^m$ with
       an inner product $\inner{~}{~}$,
 \item $D$ is a metric connection on $(\E,\inner{~}{~})$,
 \item $\varphi\colon{}TM^m\to \E$ is a bundle homomorphism
       which satisfies
       \begin{equation}\label{eq:c}
             D^{}_{X}\phi(Y)-D^{}_{Y}\phi(X)-\phi([X,Y])=0
       \end{equation}
       for vector fields $X$ and $Y$ on $M^m$.
\end{enum}
In this setting, the pull-back of the metric 
\begin{equation}\label{eq:phi-metric}
   ds^2_\phi:=\phi^*\inner{~}{~}
\end{equation}
is called the {\em $\phi$-metric}, 
which is a positive semidefinite symmetric tensor on $M^m$.

It should be remarked that as in \eqref{eq:c},
\begin{equation}\label{eq:torsion}
      T_\phi(X,Y):=D_X\phi(Y)-D_Y\phi(X)-\phi([X,Y])
\end{equation}
gives a skew-symmetric tensor on $M^m$, called
the {\it torsion tensor\/} of $D$ with respect to $\phi$.
If $(M^m,\E,\inner{~}{~},D,\phi)$ is a
coherent tangent bundle, then the pull back of the connection $D$
by $\phi$ coincides with the Levi-Civita connection.
In this sense, a coherent tangent bundle can be considered
as a generalization of Riemannian manifold.
In fact, as an application, a duality of
conformally flat manifolds
with admissible singularities is 
given in terms of
coherent tangent bundles in \cite{LUY}.

A coherent tangent bundle $(M^m,\E,\inner{~}{~},D,\phi)$
is called {\it co-orientable\/} if the vector bundle $\E$ is orientable,
namely, there exists a globally defined smooth  
section $\mu$ of $\det(\E^*)$ such that 
\begin{equation}\label{eq:mu}
   \mu(\vect{e}_1,\dots,\vect{e}_m)=\pm 1
\end{equation} 
for any orthonormal frame
$\{\vect{e}_1,\dots,\vect{e}_m\}$ on $\E$.
The form $\mu$ is determined uniquely up to a $\pm$-ambiguity.
A {\em co-orientation\/} of the coherent tangent bundle $\E$ is a 
choice of $\mu$.
An orthonormal frame $\{\vect{e}_1,\dots,\vect{e}_m\}$ is called 
{\em positive\/} with respect to the co-orientation $\mu$ 
if $\mu(\vect{e}_1,\dots,\vect{e}_m)=+1$.

We give here typical examples of coherent tangent bundles:
\begin{example}\label{ex:map}
 Let  $M^m$ be an oriented $m$-manifold
 and  $(N^m,g)$ an oriented Riemannian $m$-manifold.
 A $C^\infty$-map $f:M^m\to N^m$ induces a coherent tangent bundle
 over $M^m$  as follows: 
 Let $\E_f:=f^*TN^m$ be the pull-back of the tangent bundle $TN^m$
 by $f$.
 Then $g$ induces a positive definite metric $\inner{~}{~}$ on $\E_f$,
 and the restriction $D$ of the Levi-Civita connection of $g$
 gives a connection on $\E$ which is compatible with respect to
 the metric $\inner{~}{~}$.
 We set 
 $\phi^{}_f:=df:TM^m\to\E_f$,
 which gives the structure of the coherent tangent bundle on $M^m$.
 A necessary and sufficient condition for  a given coherent tangent
 bundle over an $m$-manifold  to be realized as
 a smooth map into the $m$-dimensional space form
 is given in \cite{SUY6}.
\end{example}

\begin{example}\label{ex:front}
 Let $(N^{m+1},g)$ be an $(m+1)$-dimensional Riemannian manifold.
 A $C^\infty$-map
 $f:M^m\to N^{m+1}$
 is called a {\it frontal\/} if for each $p\in M^m$,
 there exists a neighborhood $U$ of $p$ and a unit vector field $\nu$
 along $f$ defined on $U$ such that 
 $g\bigl(df(X),\nu\bigr)=0$ holds for any vector field $X$ on $U$
 (that is, $\nu$ is a unit normal vector field),
 and the map $\nu\colon{}U\to T_1N^{m+1}$ is a $C^{\infty}$-map,
 where $T_1N^{m+1}$ is the unit tangent bundle of $N^{m+1}$.
 Moreover, if $\nu$ can be taken to be an immersion 
 as a smooth section of $T_1N^{m+1}$ 
 for each $p\in M^m$,
 $f$ is called a {\em front} or a {\em wave front}.
 We remark that $f$ is a front if and only if
 $f$ has a lift
 $L_f:M^m\to P\bigl(T^*N^{m+1}\bigr)$
 as a Legendrian immersion,
 where $P(T^*N^{m+1})$ is a projectified cotangent bundle on $N^{m+1}$
 with the canonical contact structure.
 The subbundle $\E_f$ which consists of the vectors in
 the pull-back bundle $f^*TN^{m+1}$  
 perpendicular to $\nu$ gives a coherent tangent bundle.
 In fact, 
 $\phi_f\colon{}TM^m\ni X \mapsto df(X)\in \E_f$
 gives a bundle homomorphism.
 Let $\nabla$ be the Levi-Civita connection on $N^{m+1}$.
 Then by taking the tangential part of $\nabla$, it induces 
 a connection $D$ on $\E_f$ satisfying 
 \eqref{eq:c}.
 Let $\inner{~}{~}$ be a metric on $\E_f$ induced from 
 the Riemannian metric on $N^{m+1}$, then  $D$ is a metric connection
 on $\E_f$. 
 Thus we get a coherent tangent bundle 
 $(M^{m},\E_f,\inner{~}{~},D,\phi_f)$.
 Since the unit tangent bundle can be canonically identified with
 unit cotangent bundle, 
 the map $\nu\colon{}U\to T_1N^{m+1}$ can be identified with
 $L_f|_U$.
 A frontal $f$ is called {\it co-orientable\/} if there is a
 unit normal vector field $\nu$ globally defined on $M^m$.
 When $N^{m+1}$ is orientable, the coherent tangent bundle is 
 co-orientable if and only if so is $f$.
\end{example}

From now on, we assume that 
$(M^m,\E,\inner{~}{~},D,\phi)$ is a co-orientable 
coherent tangent bundle, 
and fix a co-orientation $\mu$ on the coherent tangent bundle.
(If $\E$ is not co-orientable, one can take a double cover 
 $\pi\colon{}\widehat M^m\to M^m$ such that the pull-back of $\E$ by $\pi$ is
 a co-orientable coherent tangent bundle over $\widehat M^m$.)
\begin{definition}\label{def:volume}
 The {\em signed $\phi$-volume form\/} $d\hat A_{\phi}$ and the 
 ({\em unsigned}) {\em $\phi$-volume form \/} $dA_{\phi}$  
 are defined  as
 \begin{equation}\label{eq:volume-form}
 \begin{aligned}
    d\hat A_\phi := \phi^*\mu = 
      \lambda_{\phi}\,du_1\wedge \dots \wedge du_m,\quad
    dA_\phi 
     := |\lambda_{\phi}|\,du_1 \wedge \dots \wedge du_m,
 \end{aligned}
 \end{equation}
 where $(U;u_1,\dots,u_m)$ is a local coordinate system of $M^m$
 compatible with the orientation of $M^m$, and 
 $\lambda_{\phi}$ 
 is the $\phi$-Jacobian function on $U$ given in \eqref{eq:jacobian}.
 Both $d\hat A_{\phi}$  and $dA_{\phi}$ are independent of the  choice of
 a positively  oriented local coordinate system $(U;u_1,\dots,u_m)$, 
 and give two globally defined $m$-forms on $M^m$.
 ($d\hat A_{\phi}$ is $C^\infty$-differentiable, 
 but $dA_{\phi}$ is only
 continuous.)
 When $M^m$ has no $\phi$-singular points, the two forms 
 coincide up to sign. Then the two sets $M^\pm_{\phi}$
 as in \eqref{eq:M_pm} are written as 
 \begin{align*}
   M^+_{\phi}&:=\bigl\{p\in M^m\setminus \Sigma_{\phi} \,;\, 
           d\hat A_{\phi}(p)=dA_{\phi}(p)\bigr\}, \\
   M^-_{\phi}&:=\bigl\{p\in M^m\setminus \Sigma_{\phi}\,;\, 
           d\hat A_\phi(p)=-dA_{\phi}(p)
 \bigr\}.
 \end{align*}
 The $\phi$-singular set $\Sigma_{\phi}$ coincides with
 the boundary  $\partial M^+_\phi=\partial M^-_\phi$.
\end{definition}

\subsection{Singularities of $2$-dimensional coherent tangent bundle}~
 
When $m=2$ and $(M^2,\E,\inner{~}{~},D,\phi)$ comes from a
front in $3$-manifold as in Example~\ref{ex:front}
(resp.\ a map into $2$-manifold as in Example~\ref{ex:map}), 
an $A_2$-point corresponds to a cuspidal edge (resp.\ a fold)
(cf.\ \cite{SUY3}, see Fig.~\ref{fig:folds}).
Though cuspidal cross caps in surfaces in $\R^3$ are not
singular points of a front, they are also $A_2$-points 
(see \cite[Remark 1.6]{SUY2}).
In this way, our definition of $A_2$-points are wider than 
the original definition of $A_2$-points on fronts
or Morin maps.
However, if $(M^2,\E,\inner{~}{~},D,\phi)$ 
is associated to a wave front, 
$A_2$-points
really corresponds to cuspidal edges on its realization as
fronts.

\begin{figure}[htb]
 \begin{center}
        \includegraphics[height=2.5cm]{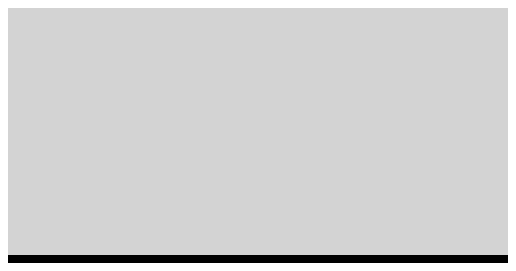}\hspace{1cm}
        \includegraphics[height=2.5cm]{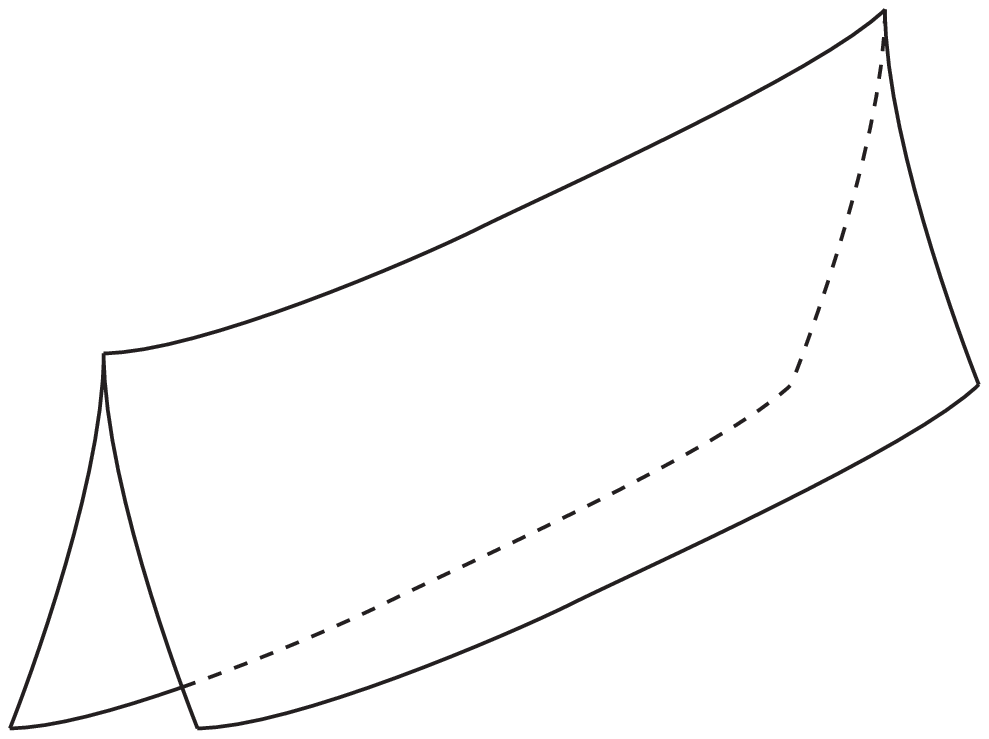}
 \end{center}
  \caption{%
   A fold in a plane (left) and  a cuspidal edge in a space (right).}
  \label{fig:folds}
\end{figure}
On the other hand,
an $A_3$-point corresponds to a swallowtail (resp.\ a cusp)
when $m=2$ and $(M^2,\E,\inner{~}{~},D,\phi)$ 
comes from a front in a $3$-manifold 
(resp.\ a map into a $2$-manifold). 
This fact was shown in 
\cite{KRSUY} (see  Fig.~\ref{fig:cusps}). 
\begin{figure}[htb]
 \begin{center}
        \includegraphics[height=2.5cm]{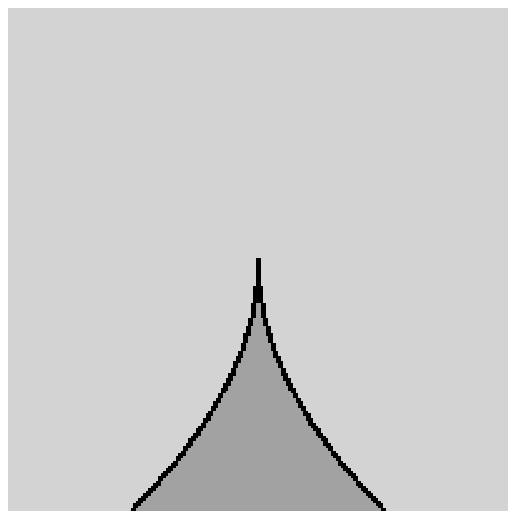}\hspace{1cm}
        \includegraphics[height=2.5cm]{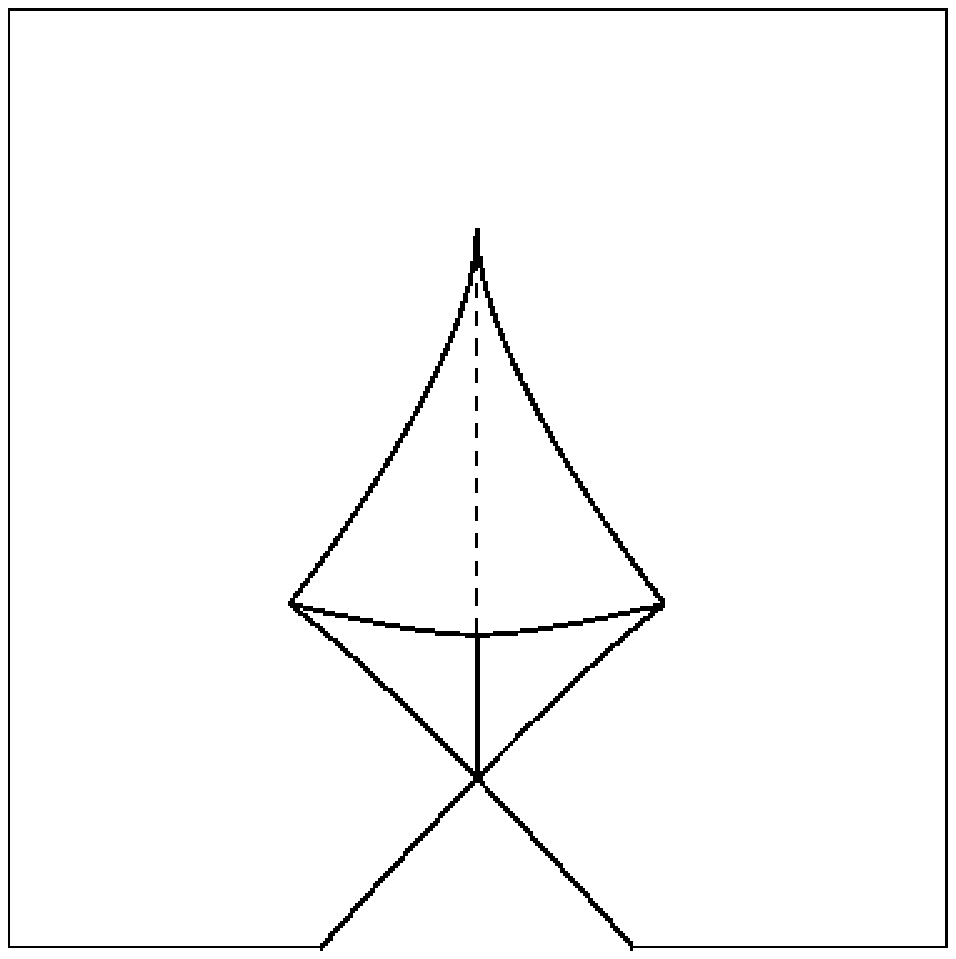}
 \end{center}
  \caption{%
    A cusp in a plane (left) and a swallowtail in a space (right).}
  \label{fig:cusps}
\end{figure}

When $\phi$ is induced from a smooth map between 
$2$-manifolds (resp.\ a front in a $3$-manifold),
it is well-known that $A_2$-points and $A_3$-points are
generic singular points of $\phi$ (cf.\ \cite{AGV}).
However, if one consider deformations of
$A_2$ or $A_3$-points,
an additional three types of singular points of type $A_3^{+}$
(lips),
$A_3^{-}$ (beaks) and of $A_4$ (butterfly) also
generically appear. 
So we mention here an intrinsic characterization of these three singular
points.
Let $(M^2,\E,\inner{~}{~},D,\phi)$ be a coherent tangent bundle
over a $2$-manifold $M^2$.
\begin{definition}[Lips]\label{def:lips}
 A $\phi$-singular point $p\in M^2$
 is called a {\em lips of $\phi$\/}
 if it satisfies the following conditions:
 \begin{enum}
  \item The rank of $\phi$ at $p$ is $1$,
  \item the exterior derivative $d\lambda_\phi$ vanishes at $p$,
  \item the Hessian matrix  of the function $\lambda_\phi$ at $p$
        is positive definite.
 \end{enum}
\end{definition}

\begin{definition}[Beaks]\label{def:beaks}
 A $\phi$-singular point $p\in M^2$
 is called a {\em beaks of $\phi$\/}
 if it satisfies the following conditions:
 \begin{enum}
  \item The rank of $\phi$ at $p$ is $1$,
  \item the exterior derivative $d\lambda_\phi$ vanishes at $p$,
  \item the Hessian matrix  of the function $\lambda_\phi$ at $p$
        is negative definite,
  \item the second derivative $\lambda''_\phi(=d\lambda'_\phi(\eta))$ 
        does not vanish 
        at $p$, where the prime means the derivative with respect to
        the null direction 
        (cf. \eqref{eq:lambda-prime} and \eqref{eq:lambda-k}).
 \end{enum}
\end{definition}

\begin{definition}[Butterfly]\label{def:b-fly}
 A non-degenerate $\phi$-singular point $p\in M^2$
 is called a {\em  butterfly of $\phi$\/}
 if it satisfies 
 \[
     \lambda'_\phi(p)=\lambda''_\phi(p)=0,\quad
     \lambda'''_\phi(p)\ne 0,
 \]
where the prime means the derivative with respect to
the null direction.
\end{definition}

\begin{remark}\label{fact:saji}
 When a coherent tangent bundle is induced from
 a front, lips, beaks and butterfly 
 correspond to cuspidal lips, cuspidal beaks and cuspidal butterfly,
 respectively 
 (Fig.~\ref{fig:front-peaks}). See  \cite{IST} and \cite{IS}.
 On the other hand, if a coherent tangent bundle is induced from
 a smooth map between $2$-manifolds, 
 these singular points are corresponding to those on the map
 (Fig.~\ref{fig:map-peaks}). See  \cite{S}.

\begin{figure}[htb]
 \begin{center}
        \includegraphics[height=3.5cm]{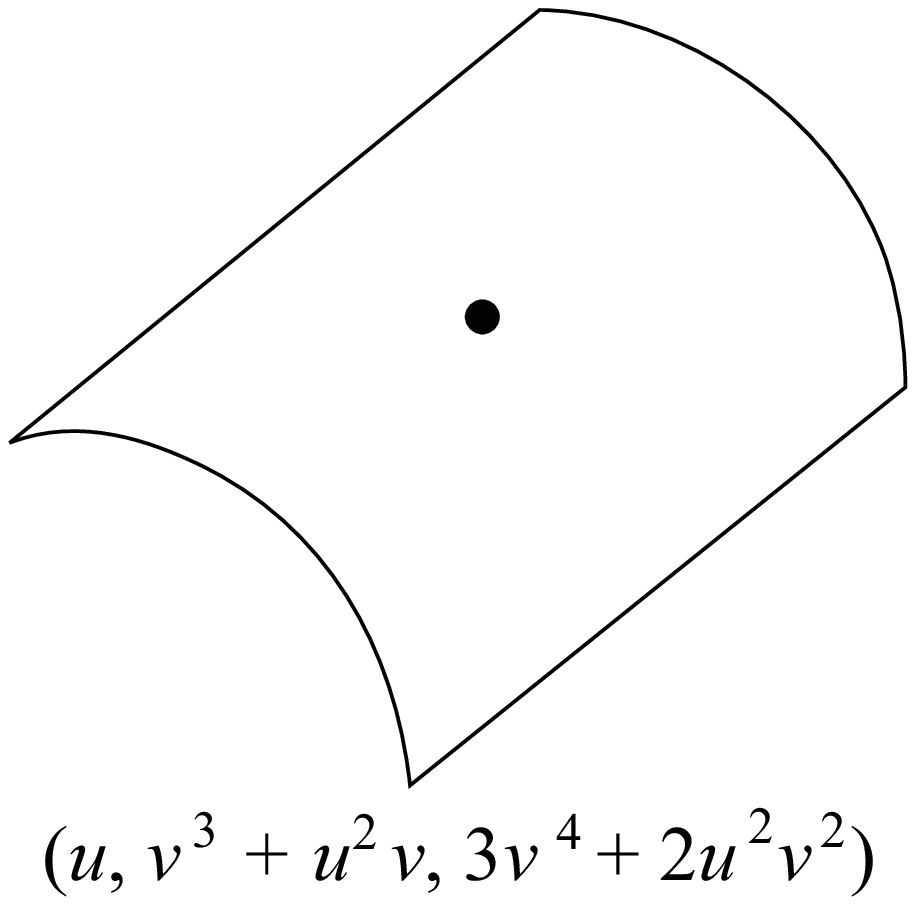}\hspace{5mm}
        \includegraphics[height=3.5cm]{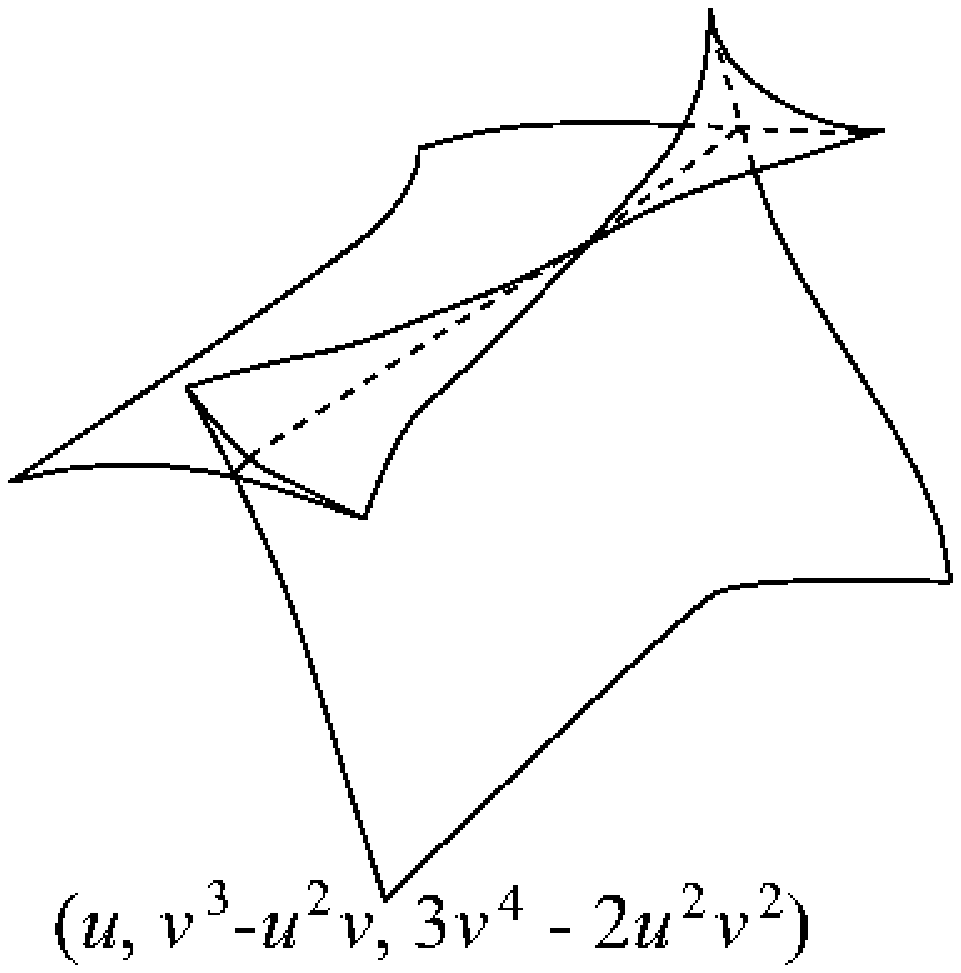}\hspace{5mm}
        \includegraphics[height=3.5cm]{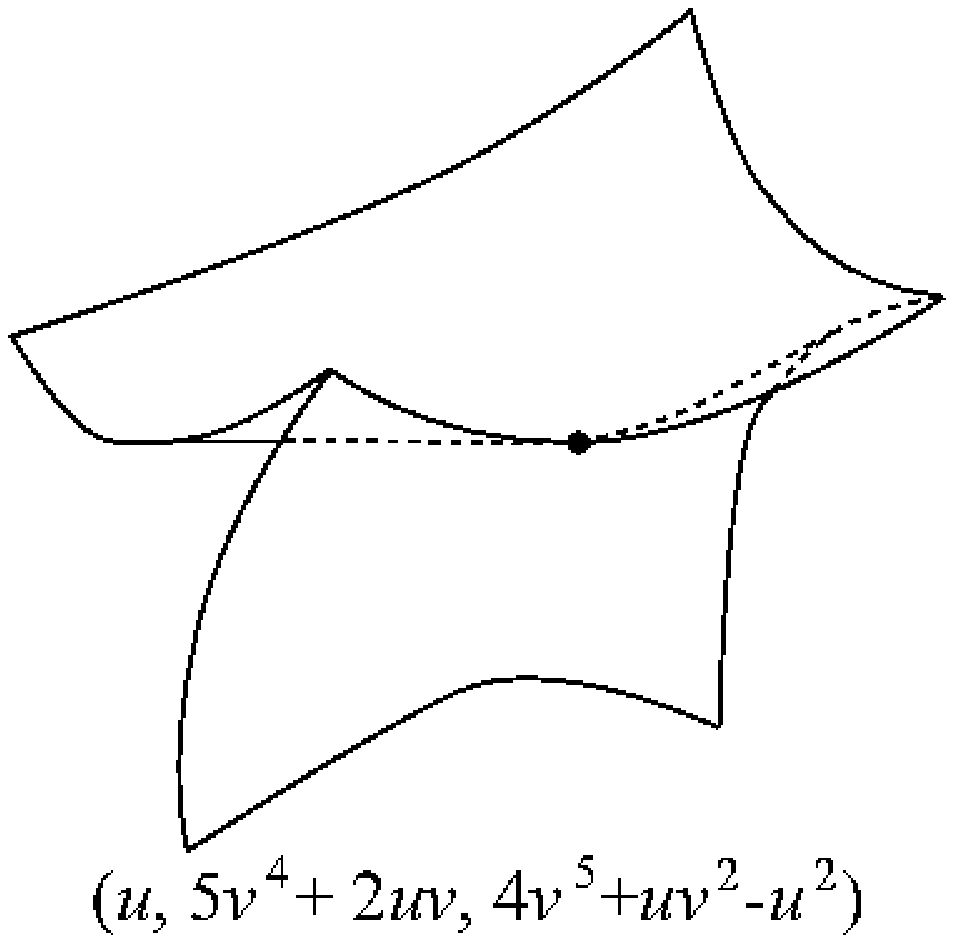}
  \caption{%
    A cuspidal lips (left), 
    a cuspidal beaks (center) and 
    a cuspidal butterfly (right) in $\R^3$.}
  \label{fig:front-peaks}
 \end{center}
 \begin{center}
        \includegraphics[height=3.2cm]{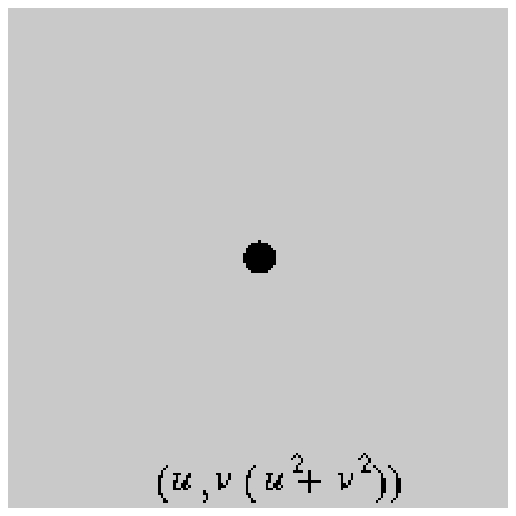}\hspace{1cm}
        \includegraphics[height=3.2cm]{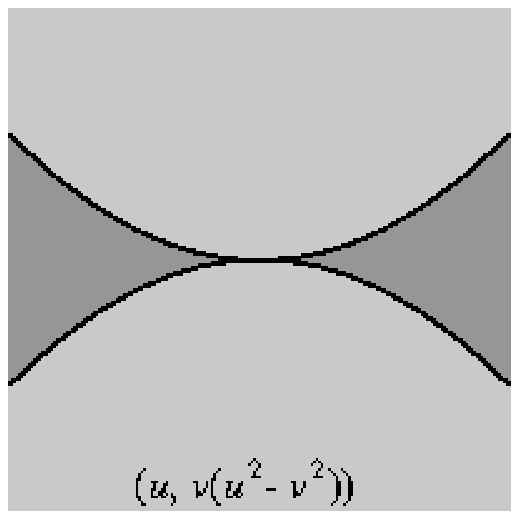}\hspace{1cm}
        \includegraphics[height=3.2cm]{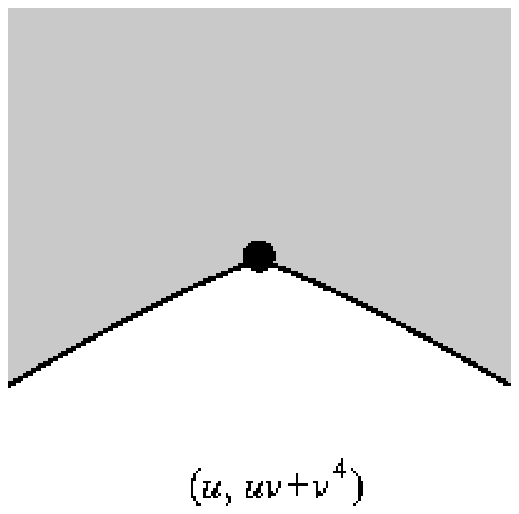}
  \caption{%
     A lips (left), a beaks (center) and a butterfly (right) in $\R^2$.}
  \label{fig:map-peaks}
\end{center}
\end{figure}
 In this way, our intrinsic formulation can give a unified treatment of
 singular points  on maps and on fronts at the same time.
\end{remark}

\subsection{Singular curvatures}
Let $(M^m,\E,\inner{~}{~},D,\phi)$ be a coherent tangent bundle
and fix  a $\phi$-singular point $p\in\Sigma_{\phi}$ which is an
$A_2$-point.
Then there exists a neighborhood $U$ of $p$ such that 
$\Sigma_{\phi}\cap U$ consists of $A_2$-points.
Now we define the {\it singular shape operator\/} as follows:
Since the kernel of $\phi_p$ is transversal to $\Sigma_\phi$ at
$p$, $\phi|_{T(\Sigma_{\phi}\cap U)}$ is injective,
where $U$ is a sufficiently small neighborhood of $p$.
Then the metric $ds^2_\phi$ is positive definite on $\Sigma_\phi\cap U$.
We take an orthonormal frame field
$e_1$, $e_2$,\dots, $e_{m-1}$
on $\Sigma_{\phi}\cap U$ with respect to $ds^2_\phi$.
Without loss of generality, we may assume that 
$(e_1,e_2,\dots,e_{m-1})$ is
smoothly extended on $U$ as an orthonormal $(m-1)$-frame field.
Then we can take a unique smooth section
$\vect{n}:U\to \E$ (called the {\it conormal vector field})
so that
$(\phi(e_1),\dots,\phi(e_{m-1}), \vect{n})$
gives a positively oriented orthonormal frame field on $\E$.
Now, we set
\begin{equation}\label{eq:sing-shape}
   S_\phi(X):=
     -\sgn\left(d\lambda_\phi\bigl(\eta(q)\bigr)\right) 
         \phi^{-1}(D_X\vect n)
     \quad (X\in T_q\Sigma_\phi,\,\,q\in \Sigma_\phi\cap U),
\end{equation}
where the non-vanishing null vector field $\eta$ is chosen so that
$(e_1,\dots,e_{m-1},\eta)$ is compatible with respect to
the orientation of $M^m$.
It holds that 
\begin{equation}\label{eq:sign2}
   \sgn\bigl(d\lambda_\phi(\eta(q))\bigr)
   =
   \begin{cases}
     \hphantom{-} 1 & 
        \mbox{if $\eta(q)$ points toward $M^+_{\phi}$},\\
     -1 & 
        \mbox{if $\eta(q)$ points toward $M^-_{\phi}$}.
   \end{cases}
\end{equation}
Since $\phi$ is injective on each tangent space of
$\Sigma_\phi$ and $D_X\vect{n}\in \phi(T\Sigma_\phi)$,
the inverse element $\phi^{-1}(D_X\vect{n})$ is uniquely determined. 
Thus we get a bundle endomorphism
$S_\phi:T\Sigma_\phi \to T\Sigma_\phi$
which is called the {\it singular shape operator\/} on $\Sigma_\phi$.

\begin{proposition}[A generalization of Theorem 1.6 in \cite{SUY1}]
\label{prop:sym}
 The definition of the singular shape operator $S_\phi$ is independent
 of the choice of an  orthonormal frame field 
 $e_1,\dots,e_{m-1}$,
 the choice of an  orientation of $M^m$,
 and the choice of a co-orientation of $\E$.
 Moreover, it holds that
 \[
   ds^2_\phi\bigl(S_\phi(X),Y\bigr)=ds^2_\phi\bigl(X,S_\phi(Y)\bigr)
        \qquad (X,Y\in T_q\Sigma_\phi,~q\in \Sigma_\phi).
 \]
\end{proposition}
\begin{proof}
 In fact, if one reverses the orientation of the basis
 $(e_1,\dots,e_{m-1})$,
 then
 $\vect{n}$ and $\eta$ change sign at the same time,
 and thus $S_{\phi}$ is unchanged.
 Similarly, 
 if we reverse the orientation of $M^m$ (resp.\ $\E$),
 then
 $\lambda_\phi$ and $\eta$ (resp.\ $\lambda_\phi$ and $\vect{n}$)
 change sign at the same time,
 and $S_\phi$ unchanged.
 
 We can get the final assertion from the
 following identity
 \begin{align*}
  ds^2_\phi&(\phi^{-1}(D_X\vect{n}),Y)=
   \inner{D_X\vect{n}}{\phi(Y)}
    =-\inner{\vect{n}}{D_X\phi(Y)}\\
   &=-\inner{\vect{n}}{D_Y\phi(X)+\phi\bigl([X,Y]\bigr)}\\
   &=-\inner{\vect{n}}{D_Y\phi(X)}
    = \inner{D_Y\vect{n}}{\phi(X)}
    =ds^2_\phi\bigl(\phi^{-1}(D_Y\vect{n}),X\bigr),
 \end{align*}
 where $X$ and $Y$ are both vector fields on $\Sigma_\phi$.
\end{proof}
Hence $S_\phi$ is  symmetric with respect to $ds^2_{\varphi}$.
\begin{definition}\label{def:singular-curvature}
 Let $p\in\Sigma_\phi$ be an $A_2$-point of $\phi$.
 Then
 \begin{equation}\label{eq:kappa}
  \kappa_\phi(X):=ds^2_\phi(S_\phi(X),X)/ds^2_\phi(X,X),
   \qquad (X\in T_p\Sigma_{\phi}\setminus\{0\})
 \end{equation}
 is called the {\it $\phi$-singular normal curvature\/}
 at $p$ with respect to the direction $X$.
 The eigenvalues of $S_\phi$ are called the 
 {\it $\phi$-singular principal curvatures\/}, 
 which give the critical values of the singular normal curvature on
 $T_p\Sigma_\phi$.
\end{definition}
When $m=2$, the $\phi$-singular principal curvature 
is called (simply) the {\it $\phi$-singular curvature},
which is also denoted by $\kappa_\phi$.
This definition of the singular curvature
is the same as in \cite[(1.7)]{SUY1} and \cite[(1.6)]{SUY2}.
More precisely, $\kappa_\phi$ is computed as follows:
Let $p\in\Sigma_\phi$ be an $A_2$-point of $\phi$.
Then the $\phi$-singular set $\Sigma_\phi$ is parametrized 
by a regular curve $\gamma(t)$ ($t\in I\subset\R$)
on $M^2$ on a neighborhood of $p$, 
and $\gamma(t)$ is an $A_2$-point of $\phi$ for each $t\in I$.
Since $\dot\gamma(t)$ ($\dot{~}=d/dt$) is not a null-direction,
$\phi\bigl(\dot\gamma(t)\bigr)\neq 0$.
Take a section $\vect{n}(t)$ of $\E$ along $\gamma$ such that
$\{\phi(\dot\gamma)/|\phi(\dot\gamma)|,\vect{n}\}$
gives a positive orthonormal frame field on $\E$ along $\gamma$,
where
$|\phi(\dot\gamma)|=\inner{\phi(\dot\gamma)}{\phi(\dot\gamma)}^{1/2}$.
Then we have
\begin{equation}\label{eq:singular-curvature-2}
  \kappa_\phi(t) 
   :=\kappa_{\phi}\bigl(\dot\gamma(t)\bigr)
    = -\sgn \left(d\lambda_\phi\bigl(\eta(t)\bigr)\right)
    \frac{\inner{D_{d/dt}\vect{n}(t)}{\phi\bigl(\dot\gamma(t)\bigr)}}{%
            |\phi\bigl(\dot\gamma(t)\bigr)|^2},
\end{equation}
where $\eta(t)$ is a null-vector field along $\gamma(t)$ such that
$\{\dot\gamma(t),\eta(t)\}$ is compatible with the orientation of $M^2$.
By \eqref{eq:sign2}, 
it holds that 
\begin{equation}\label{eq:sign3}
   \sgn\bigl(d\lambda_\phi(\eta(t))\bigr)
   =\begin{cases}
    \hphantom{-}1 
       & \mbox{if  $M^+_{\phi}$ lies on 
               the left-hand side of $\gamma$},\\
     -1 
       & \mbox{if  $M^-_{\phi}$ lies on 
               the left-hand side of $\gamma$}.
\end{cases}
\end{equation}

\subsection{Behavior of the singular curvatures}
We now prove the following:
\begin{theorem}[A generalization of {\cite[Corollary 1.14]{SUY1}}]
\label{thm:infty}
 Let $p\in\Sigma_{\phi}$ be a $\phi$-singular point of 
 a coherent tangent bundle $(M^m,\E,\inner{~}{~},D,\varphi)$, 
 and assume $p$ is non-degenerate but not an $A_2$-point of $\phi$.
 Take a regular curve 
 $\gamma\colon [0,1]\ni t\mapsto \gamma(t)\in\Sigma_\phi$ such that
 such that $\gamma\bigl((0,1]\bigr)$ consists only of $A_2$-points 
 of $\phi$ and $\gamma(0)=p$.
 Then one of the $\phi$-singular principal curvatures along
 $\gamma(t)$ diverges to $-\infty$.
\end{theorem}
\begin{proof}
 We can take a local coordinate system $(U;u_1,\dots,u_{m})$ on $M^m$
 at $p$ satisfying the following properties:
 \begin{enum}
 \item\label{item:coord-1}
       $(u_1,\dots,u_{m})$ is compatible with respect to the
       orientation of $M^m$,
 \item\label{item:coord-2}
       the $\phi$-singular submanifold is characterized by the zeros of the 
       last coordinate function, that is,
       $\Sigma_\phi\cap U=\{u_m=0\}$.
 \item\label{item:coord-3} 
       the vector field 
        $\eta:=\frac{\partial}{\partial u_1} 
                 + \delta \frac{\partial}{\partial u_m}$
       gives the null direction,
       where $\delta:=\delta(u_1,\dots,u_{m-1})$
       is a $C^\infty$-function satisfying $\delta(p)=0$.  
       (In fact, $\eta(p)$ must tangent to the singular manifold,
       since $p$ is not an $A_2$-point.)
 \end{enum}
 For the sake of simplicity,  we set
 \begin{equation}\label{eq:notation}
     \partial_j:=\frac{\partial}{\partial u_j},\quad
     \phi_j:=\phi(\partial_j),\quad D_j:=D_{\partial_j}\qquad 
     (j=1,\dots,m).
 \end{equation}
 Since $\gamma(t)$ ($t\in (0,1]$) is an $A_2$-point of $\phi$,
 \ref{item:coord-2} and \ref{item:coord-3} imply that
 $\delta(t):=\delta\bigl(\gamma(t)\bigr)$
 does not vanish for
 $t\in (0,1]$.
 Then $\{\partial_1,\dots,\partial_{m-1},\sgn(\delta)\eta\}$
 forms a positive frame field on $TM^m$ along $\gamma(t)$ ($t\neq 0$).
 By definition, the $\phi$-singular normal curvature
 along $\gamma$ with respect to $\partial_1$ is given by 
 \begin{align}
  \label{eq:kappa-phi-t}
  \kappa_\phi(t)\biggl(
     :&=\kappa_\phi(\partial_1)\biggr)
       =-\sgn \left(d\lambda_\phi\bigl(\sgn(\delta)\eta\bigr)\right) 
       \frac{\inner{D_1\vect{n}}{\phi_1}}{|\phi_1|^2}\\
      &=-\sgn(\delta)
       \sgn \bigl(d\lambda_\phi(\eta)\bigr) 
       \frac{\inner{\vect{n}}{D_1\phi_1}}{|\phi_1|^2}.
  \nonumber
 \end{align}
 Since $\eta$ is a null vector, it holds that
 \begin{equation}\label{eq:null}
      \phi(\eta)=\phi_1+\delta \phi_m=0
 \end{equation}
 on $\Sigma_\phi$.
 In particular, since $\eta=\partial_1$ at $p$, we have that 
 \begin{equation}\label{eq:phi-m}
     \phi_m(p)\ne 0.
 \end{equation}
 Differentiating \eqref{eq:null} by $u_1$, we have
 \begin{equation}\label{eq:null2}
    D_1\phi_1+\delta_1 \phi_m+\delta D_1 \phi_m=0
     \qquad \left(\delta_1:=\frac{\partial \delta}{\partial u_1}\right).
 \end{equation}
 We can identify $\bigwedge_{j=1}^{m-1}\E_q\cong \E_q$
 for each $q\in M^m$ by using the inner product on $\E$,
 and then we set
 \begin{equation}\label{eq:use_next_paper}
      \vect{n}=\frac{\phi_1\wedge\dots \wedge \phi_{m-1}}
               {|\phi_1\wedge\dots \wedge \phi_{m-1}|},
 \end{equation}
 which gives a conormal vector field such that
 $\{\phi_1,\dots,\phi_{m-1}, \vect{n}\}$
 is a positive frame field on $\E$ along $\gamma(t)$
 ($t\neq 0$).

 By \eqref{eq:null}, we have $\phi_1=-\delta \phi_m$.
 Then
 \begin{equation}\label{eq:n-red}
   \vect{n}=
      -\frac{\delta}{|\delta|}\cdot 
      \frac{\phi_m\wedge\phi_2\wedge \dots \wedge \phi_{m-1}}
      {|\phi_m\wedge\phi_2\wedge \dots \wedge \phi_{m-1}|}
      =(-1)^{m-1} \sgn(\delta)
      \frac{\phi_2\wedge \dots \wedge \phi_{m}}
      {|\phi_2\wedge \dots \wedge \phi_{m}|}.
 \end{equation}
 Substituting \eqref{eq:null} and \eqref{eq:null2} into 
 \eqref{eq:kappa-phi-t},  we have that
 \begin{align*}
  \kappa_\phi(t)&=
     -\sgn(\delta)\sgn \bigl(d\lambda_\phi(\eta)\bigr) 
   \frac{\inner{\vect{n}}{\delta_1 \phi_m+\delta D_1 \phi_m}}
        {|\delta \phi_m|^2} \\
    &=-    \sgn \bigl(d\lambda_\phi(\eta)\bigr) 
        \frac{\inner{\vect{n}}{D_1 \phi_m}}{|\delta|\, |\phi_m|^2}. 
 \end{align*}
 Set
 \[
    \Delta := 
         \frac{(-1)^{m-1}\sgn \bigl(d\lambda_\phi(\eta)\bigr)}
            {|\delta|\, |\phi_m|^2|\phi_2\wedge \dots \wedge \phi_{m}|}
       \qquad\text{and}\qquad
     \lambda_m:=\frac{\partial\lambda_\phi}{\partial u_m}.
 \]
 Then substituting \eqref{eq:n-red} into the above equation and 
 noticing that $\phi_1$ is proportional to $\phi_m$, we have
 \begin{align*}
  \frac{\kappa_\phi(t)}{\Delta}&=
      \inner{\phi_2\wedge \dots \wedge \phi_{m}}{D_m\phi_1}\\
  &= \partial_m
      \inner{\phi_2\wedge\dots\wedge\phi_m}{\phi_1}
      -\inner{D_m(\phi_2\wedge\dots\wedge\phi_m)}{\phi_1}\\
  &= (-1)^{m-1}\partial_m\mu(\phi_1,\dots,\phi_m)
      -\inner{\phi_2\wedge\dots\wedge D_m\phi_m}{\phi_1}\\
  &= (-1)^{m-1}\lambda_m
        +
        \inner{\phi_2\wedge\dots\wedge D_m(\delta\phi_1)}{\phi_1}\\
  &= (-1)^{m-1}\lambda_m
        +
        \delta\inner{\phi_2\wedge\dots\wedge D_m \phi_1}{\phi_1}.
 \end{align*}
 Hence 
 \[
     \kappa_\phi(t)=
         \frac{\sgn \bigl(d\lambda_\phi(\eta)\bigr)\sgn(\delta)}
            {|\delta|\, |\phi_m|^2|\phi_2\wedge \dots \wedge \phi_{m}|}
         \lambda_m + 
         \text{(a bounded term)},
 \]
 because of \eqref{eq:phi-m}.
 Since
 $d\lambda_\phi(\eta)=d\lambda_\phi(\partial_1+\delta \partial_m)
           =\delta \lambda_m$,
 we have that
 \[
   \kappa_{\phi}(t)=
   -\frac{|\lambda_m|}{%
           |\delta|\,|\phi_m|^2|\phi_2\wedge \dots \wedge \phi_{m}|} 
   +\mbox{(a bounded term)}. 
 \]
 Since $p$ is non-degenerate and $\Sigma_\phi=\{u_m=0\}$, 
 $\lambda_m\ne 0$ holds. 
 Thus, $\kappa_{\phi}(t)$ tends to $-\infty$ as $t\to +0$.
 Hence, at least one of the singular principal curvatures tends to
 $-\infty$.
\end{proof}
\subsection{Frontal bundles}
At the end of  this section, 
we give a definition of frontal bundles
as an intrinsic characterization of
wave fronts in space forms.

Let $M^m$ be an  oriented $m$-manifold and 
$(M^m,\E,\inner{~}{~},D,\phi)$ 
a co-orientable coherent tangent bundle over $M^m$.
If there exists another  bundle homomorphism $\psi:TM^m\to \E$
such  that $(M^m,\E,\inner{~}{~},D,\psi)$
is also a coherent tangent bundle and
the pair $(\phi,\psi)$ of bundle homomorphisms satisfies a compatibility
condition
\begin{equation}\label{eq:compati}
  \inner{\phi(X)}{\psi(Y)}=\inner{\phi(Y)}{\psi(X)},
\end{equation}
then $(M^m,\E,\inner{~}{~},D,\phi,\psi)$ is called a {\it frontal bundle}.
The bundle homomorphisms $\phi$ and $\psi$ are called
the {\em first homomorphism\/} and the {\em second homomorphism},
respectively.
We set
\begin{align*}
 \first(X,Y)&:=ds^2_{\phi}(X,Y)=\inner{\phi(X)}{\phi(Y)}, \\
 \second(X,Y)&:=-\inner{\phi(X)}{\psi(Y)},\\
 \third(X,Y)&:=ds^2_{\psi}(X,Y)= \inner{\psi(X)}{\psi(Y)} 
\end{align*}
for $X,Y\in T_pM^m$ ($p\in M^m$),  and we call them 
{\it the first, the second and the third fundamental forms},
respectively. 
They are all symmetric covariant tensors on $M^m$.

\begin{definition}\label{def:front}
 A frontal bundle $(M^m,\E,\inner{~}{~},D, \phi,\psi)$
 is called  a {\it front bundle\/} if
 \begin{equation}\label{eq:front}
  \Ker(\phi_p)\cap \Ker(\psi_p)=\{0\}
 \end{equation}
 holds for each  $p\in M^m$.
\end{definition}

\begin{example}\label{ex:front-bundle}
 Let $\bigl(N^{m+1}(c),g\bigr)$ be 
 an $(m+1)$-dimensional space form,
 that is, a complete Riemannian $(m+1)$-manifold of constant curvature
 $c$,
 and denote by $\nabla$ the Levi-Civita connection 
 on $N^{m+1}(c)$.
 Let $f:M^m\to N^{m+1}(c)$ be a co-orientable frontal.
 Then there exists a globally defined unit normal vector field $\nu$.
 Since the coherent tangent bundle $\E_f$ given in
 Example~\ref{ex:front} is orthogonal to $\nu$,
 we can define a bundle homomorphism
 \[
    \psi_f:T_pM^m\ni X\longmapsto \nabla_X\nu \in \E_p
            \qquad (p\in M^m).
 \]
 Then $(M^m,\E_f,\inner{~}{~},D,\phi_f,\psi_f)$ is a frontal bundle.
 Moreover, this is a front bundle in the sense of
 Definition~\ref{def:front} if and only if $f$ is a front,
 which is equivalent to $\first+\third$ being positive definite.
 As a fundamental theorem for hypersurfaces, 
 the integrability condition for a
 given frontal bundle to be realized as a frontal 
 in a space form is given in \cite{SUY6}.
\end{example}

We fix a front bundle $(M^m,\E,\inner{~}{~},D,\phi,\psi)$
over an $m$-dimensional manifold $M^m$.
\begin{definition}\label{def:ext}
 When $p\in M^m$ is not a singular point of $\phi$, we define
 \begin{equation}\label{eq:ext}
      K^{\ext}(X\wedge Y):=
      \frac{
        \second(X,X) \second(Y,Y)
            -\second(X,Y)^2
      }{
         I(X,X) I(Y,Y)
             -I(X,Y)^2
      }
     \qquad (X,Y\in T_pM^m),
 \end{equation}
 which is called the {\em extrinsic curvature\/} at $p$ 
 with respect to the $X\wedge Y$-plane in $T_pM^m$.
\end{definition}

If a front bundle $(M^m,\E,\inner{~}{~},D,\phi,\psi)$
is induced from a front in $N^{m+1}(c)$,
then it holds that
 \begin{equation}\label{eq:ext2}
      K^{\ext}(X\wedge Y)=K(X\wedge Y)+c
     \qquad (X,Y\in T_pM^m),
 \end{equation}
where $K(X\wedge Y)$ is the sectional curvature
at each $\phi$-regular point $p$ of $M^m$. 

As a generalization of \cite[Theorem 3.1]{SUY1},
relationships between singular principal curvatures
and  $K^{\ext}$ of wave fronts in space forms are
investigated in \cite{SUY6}.

\section{Four Gauss-Bonnet formulas on surfaces}
\label{sec:gauss-bonnet}
In this section, we give four Gauss-Bonnet formulas
on a given front in $3$-dimensional space forms,
and will point out several remarkable consequences of them.

\subsection{The Gauss-Bonnet formulas for smooth maps}
Let  $M^m$ be an oriented $m$-manifold and 
$(N^m,g)$ an oriented Riemannian $m$-manifold.
As in Example \ref{ex:map},
a $C^\infty$-map $f\colon{}M^m\to N^m$ induces a coherent tangent bundle
$(M^m,\E_f,\inner{~}{~},D,\phi_f=df)$ over $M^m$.
In this setting, an $A_k$-point 
(cf.~Definitions~\ref{def:a2-point} and \ref{def:a3-point})
coincides with an $A_{k}$-Morin singular point of $f$
(see \cite{SUY3}).

Now we restrict our attention to the case $m=2$:
An $A_2$-point (resp.\ $A_3$-point) of $\phi_f$ on $M^2$  is called 
a {\it fold\/} (resp. a {\it cusp}); namely,
a fold (resp.\ a cusp) is right-left equivalent to the map
\[
   \R^2\ni (u,v)\longmapsto (u^2,v)\in \R^2 \quad 
   \bigl(\mbox{resp.\ }  \R^2\ni (u,v)
   \longmapsto (uv+v^3,u)\in \R^2 \bigr)
\]
at the origin.
Here, two map germs $f_{i}:(\R^m,p)\to(\R^n,q)$ $(i=1,2)$
are right-left equivalent if
there exist diffeomorphism germs
$\xi_1:(\R^m,p)\to(\R^m,p)$ and 
$\xi_2:(\R^n,q)\to(\R^n,q)$ such that
$\xi_2\circ f_1=f_2\circ\xi_1$ holds.

The $\phi_f$-singular curvature $\kappa_f(p)$ of
an $A_2$-point $p$ of $\phi_f$ (i.e.\ a fold)
is called the {\it singular curvature\/} at a fold.
The following assertion follows immediately.
\begin{proposition}\label{prop:geod}
 Let $f:M^2\to N^2$ be a $C^\infty$-map and $p$ a fold singular point.
 Suppose that $\gamma(t)$ is a regular curve which parametrizes the
 singular set so that
 $f(M^2)$ lies on the left-hand side of $f\circ\gamma$.
 Then the singular curvature $\kappa_f(p)$ at $p$ is equal to
 the geodesic curvature of $f\circ \gamma$.
\end{proposition}

\begin{example}\label{ex:parabola}
 We set 
 $f_\epsilon(u,v):=(u^2+\epsilon v^2,v)$
 $(\epsilon:=\pm 1)$.
 If $\epsilon=1$ (resp. $\epsilon=-1$),
 then all singular points consist of folds with positive 
 (resp.\ negative) singular curvature, see Fig.\ \ref{fig:pn-folds}.
\end{example}

\begin{figure}[h]
 \begin{center}
        \includegraphics[height=3cm]{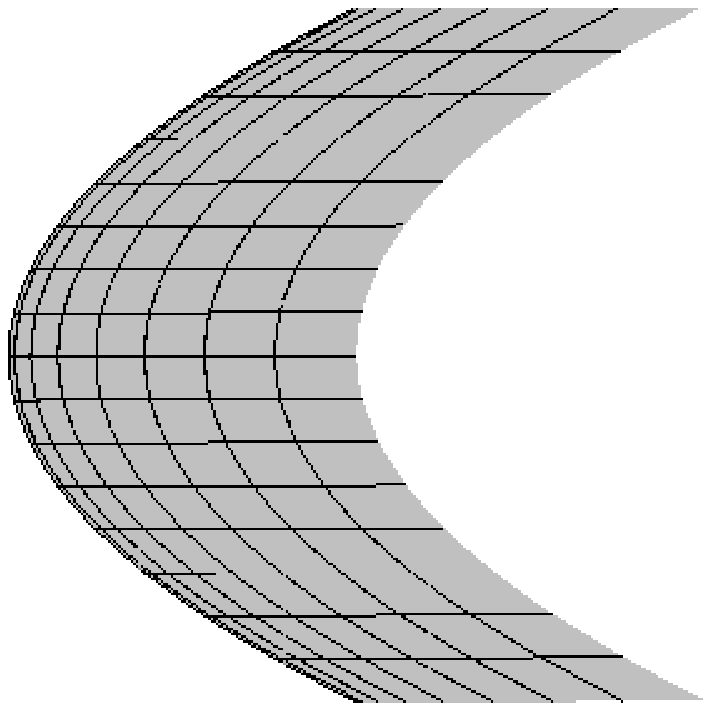}\hspace{1cm}
        \includegraphics[height=3cm]{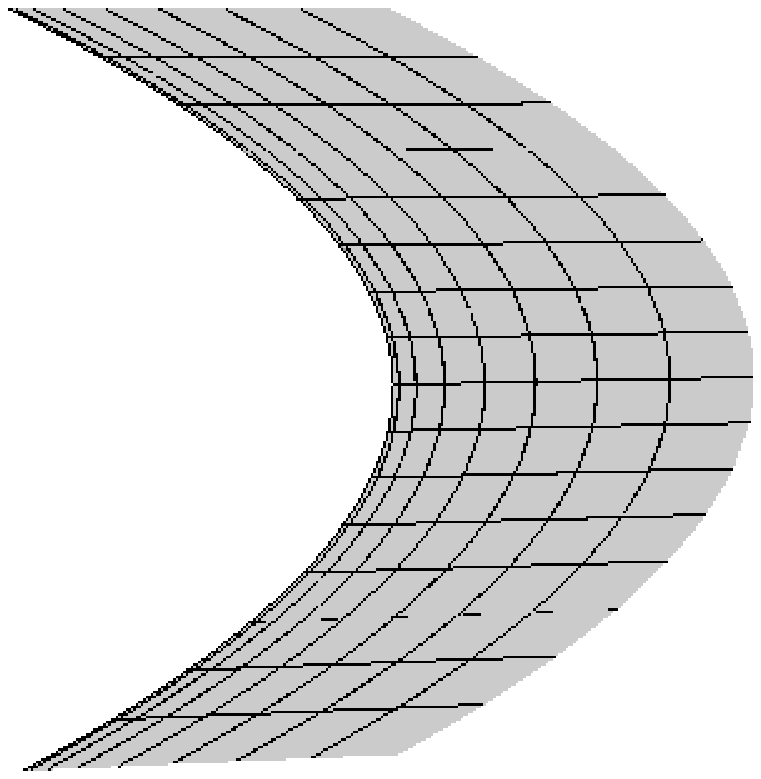}
 \end{center}
  \caption{%
  The images of $f_1$ and $f_{-1}$ 
  (the folds are the left edges in both figures).}
  \label{fig:pn-folds}
\end{figure}
Now we recall the following two Gauss-Bonnet formulas
on a given coherent tangent bundle:
\begin{fact}[\cite{SUY1,SUY2}]
\label{fact:suy}
 Let  $(M^2,\E,\inner{~}{~},D,\phi)$ be a coherent tangent bundle 
 over a compact oriented $2$-manifold $M^2$, and suppose that
 the $\phi$-singular set $\Sigma_\phi$ consists of $A_2$-points  
 and  $A_3$-points.
 We denote by $K$ the Gaussian curvature of the
 induced metric $ds^2=\phi^*\langle\,,\,\rangle$.
 Then it holds that
 \begin{align}
  (\chi_{\E}^{}=)&\frac1{2\pi}\int_{M^2}K\, d\hat A_\phi=
        \chi(M^+_\phi)-\chi(M^-_\phi)+S^+_\phi-S^-_\phi, 
  \label{eq:B}\\
  2\pi\chi(M^2)&=
  \int_{M^2}K\, dA_\phi+2\int_{\Sigma_\phi} \kappa_\phi\, d\tau_\phi, 
  \label{eq:A}
 \end{align}
where $d\tau_\phi$ is the 
length element on the  $\phi$-singular 
set with respect to $ds^2_\phi$,
and $S^+_\phi$ and $S^-_\phi$ are 
the numbers of positive and negative 
$A_3$-points of $\phi$, respectively  
{\rm(}see {\rm\cite[Figure 4]{SUY1}} and
also {\rm\cite[Definition 2.9]{SUY2}}{\rm)}.
 On the other hand,
 $\chi_{\E}^{}$ is the Euler characteristic of the
 $\SO(2)$-vector bundle $\E$.
\end{fact}

\begin{remark}\label{rem:GB}
 Let $(\E,\inner{~}{~})$ be an oriented vector bundle of
 rank $2$ with a metric over a compact 
 oriented $2$-manifold $M^2$,
 and $D$ a metric connection.
 Suppose that there exists a bundle homomorphism
 $\phi:TM^2\to \E$.
 In the first section, $A_2$ and $A_3$-points 
 are defined for such an
 arbitrary bundle homomorphism $\phi$
 and 
 {\it \eqref{eq:B} holds without assuming the compatibility
 condition \eqref{eq:c}}, as pointed out in \cite{SUY5}:
 For an oriented orthonormal frame field $\{\vect{e}_1,\vect{e}_2\}$
 of $\E$ defined on $U\subset M^2$,
 there is a unique $1$-form $\omega$ on $U$ such that
\begin{equation}\label{eq:conn-form}
    D_X\vect{e}_1=-\omega(X)\vect{e}_2,\qquad 
    D_X\vect{e}_2=\omega(X)\vect{e}_1.
\end{equation}
 Then $d\omega$ does not depend on the choice of 
 $\{\vect{e}_1, \vect{e}_2\}$, and
 there is a $C^\infty$-function $K_{\phi,D}$
 on $M^2\setminus \Sigma_\phi$ such that
 \begin{equation}\label{eq:K}
    d\omega=K_{\phi,D}\, d\hat A_{\phi}.
 \end{equation}
 We call $K_{\phi,D}$ the {\it Gaussian curvature\/} of $D$ with respect
 to $\phi$. If $\E$ is a coherent tangent bundle, 
 the identity \eqref{eq:c} implies that $K_{\phi,D}$
 coincides with the Gaussian curvature of the metric
 $ds^2_\phi=\phi^*\inner{~}{~}$.
 Thus formulas \eqref{eq:B} and \eqref{eq:A}
 still hold for $K=K_{\phi,D}$,
 without assuming \eqref{eq:c}. 
 Although \eqref{eq:A} depends on $D$,
 \eqref{eq:B} is independent of the choice of metric connections.
\end{remark}

\begin{remark}\label{rem:peaks}
 Under the assumption that $\E$ is a coherent tangent bundle,
 the identities \eqref{eq:B} and \eqref{eq:A} hold even if $\phi$ admits
 wider class of singular points  called `peaks'.
 (cf.\ \cite[Definition 1.10]{SUY1} and also \cite[Definition 2.1]{SUY2}).
 $A_3$-points, lips (see Definition~\ref{def:lips}),
 beaks (see Definition~\ref{def:beaks}),
 butterflies (see Definition~\ref{def:b-fly}) are all examples of peaks.
 For a coherent tangent bundle over a compact oriented $2$-manifold
 whose singular points are at most peaks,
 the formulas \eqref{eq:B} and \eqref{eq:A} hold
 (see \cite[Theorem 2.3]{SUY1} and \cite[Theorem~B]{SUY2}),
 here $S^+_{\phi}$ and $S^-_{\phi}$ in \eqref{eq:B} should be 
 replaced by the numbers of positive (resp.\ negative) peaks
 (cf.\ \cite[Definition 2.1]{SUY2}).
\end{remark}

The formula \eqref{eq:B} in our situation
induces Quine's formula.
\begin{proposition}[Quine \cite{Q}]
\label{prop:quine}
 Let $M^2$ and $N^2$ both be compact oriented connected $2$-manifolds, 
 and $f:M^2\to N^2$ be a $C^\infty$-map whose singular set consists of
 folds and cusps.
 Then the topological degree of $f$ satisfies
 \[
   \deg(f)\chi(N^2)=\chi(M^+_f)-\chi(M^-_f) +S^+_f-S^-_f,
 \]
 where $M^+_f$ {\rm(}resp.\ $M^-_f${\rm)} is the set of
 regular points at which $f$ preserves {\rm(}resp. reverses{\rm)} 
 the  orientation,
 and $S^+_f$ {\rm(}resp. $S^-_f${\rm)} is the number of positive cusps
 {\rm(}resp.\ the number of negative cusps{\rm)}.
\end{proposition}

On the other hand, the formula \eqref{eq:A} induces the following new
formula.
\begin{proposition}\label{prop:id}
 Let $(N^2,g)$ be an oriented Riemannian  $2$-manifold, and
 $M^2$ a compact oriented $2$-manifold.
 Let  $f:M^2\to N^2$ be a $C^\infty$-map whose singular set
 consists of folds and cusps.
 Then the total singular curvature $\int_{\Sigma} \kappa\, d\tau$
 with respect to the length element $d\tau$  
 {\rm(}with respect to $g${\rm)} on the singular set $\Sigma$
 is bounded, and satisfies the following identity 
 \[
     2\pi\chi(M^2)=
        \int_{M^2}(\widetilde K\circ f) \,|f^*dA_g|+
         2\int_{\Sigma} \kappa\, d\tau, 
 \]
 where $\widetilde{K}$ is the Gaussian curvature function on $(N^2,g)$,
 and  $|f^*dA_g|$ is the pull-back of the Riemannian measure of $(N^2,g)$.
\end{proposition}

In particular, if $(N^2,g)$ is the Euclidean
plane $(\R^2,g_0)$, we have the following
classical result: 
\begin{corollary}[Levine \cite{L}]\label{cor:id} 
 Let $M^2$ be a compact oriented $2$-manifold, and
 $f:M^2\to \R^2$ a $C^\infty$-map whose singular set
 consists of folds and cusps.
 Let $C_1,\dots,C_r$ be the disjoint union of the closed regular curves on
 $M^2$ such that the singular set of
 $f$ is $C_1\cup\dots \cup C_r$.
 Suppose that we give an orientation to each $f(C_j)$ 
 {\rm(}$j=1,\dots,r${\rm)}
 so that the image $f(M^2)$ lies on the left-hand  side of $f(C_j)$.
 Then the rotation indices $I(C_j)$ 
 {\rm(}which takes values in the set of half-integers{\rm)} of
 $f(C_j)$ {\rm(}$j=1,\dots,r${\rm)} as fronts satisfy
 \[
    \frac{\chi(M^2)}2=I(C_1)+\ldots +I(C_r).
 \]
\end{corollary}

\subsection{The  Gauss-Bonnet formulas for wave fronts}
We use the same notation as in the previous sections.
Let $N^3(c)$ be a $3$-dimensional complete Riemannian manifold with
constant curvature $c$.
Here, we do not assume that $N^3(c)$ is simply connected,
and we use the notation
$\widetilde N^3(c)$ when it is in fact
assumed to be  simply connected, like in
the previous section.

Let $M^2$ be an oriented closed $2$-manifold and
$f:M^2\to N^3(c)$ a front.
Then the front bundle
$(M^2,\E_f,\inner{~}{~},D,\phi=\phi_f,\psi=\psi_f)$ is defined as in 
Example~\ref{ex:front-bundle}.
We fix a positively oriented local coordinate system $(U;u,v)$ on
$M^2$.
Denote by $\lambda$ (resp.\ $\lambda_\#$) the $\phi$-Jacobian 
function (resp.\ $\psi$-Jacobian function)
as in \eqref{eq:jacobian}:
\begin{equation}\label{eq:2-lambda}
     \lambda:=\lambda_{\phi}=\mu(\phi_u,\phi_v)
     \qquad\text{and}\qquad
     \lambda_\#:=\lambda_{\psi}=\mu(\psi_u,\psi_v),
\end{equation}
where 
\[
   \phi_u=\phi(\partial /\partial u),\quad
   \phi_v=\phi(\partial /\partial v),\quad
   \psi_u=\psi(\partial /\partial u),\quad
   \psi_v=\psi(\partial /\partial v),
\]
and $\mu$ is the volume form of $\E_f$ giving 
co-orientation of $\E_f$.
Then the $\phi$-singular set (resp.\ $\psi$-singular set)
\[
     \Sigma:=\Sigma_{\phi}\qquad
     (\text{resp.\ }\Sigma_\#:=\Sigma_{\psi})
\]
is expressed as
\[
     \Sigma\cap U = \{p\in U\,;\,\lambda(p)=0\}\qquad
     (\text{resp.\ }
     \Sigma_\#\cap U = \{p\in U\,;\,\lambda_\#(p)=0\}).
\]

\begin{remark}\label{rmk:2}
If $N^3(c)=\R^3$, $\Sigma_\#$ is the singular set of the
 Gauss map $\nu:M^2\to S^2$ of $f$.
 When $N^3(c)=S^3$,
 $\Sigma_\#$ is the singular set of the map induced by the unit normal
 vector field  $\nu:M^2\to S^3$,
 where $S^3$ is the unit $3$-sphere.
 This unit normal vector corresponds to the dual map of $f$
 induced by a duality between $S^3$ and itself.
 When $N^3(c)$ is the hyperbolic $3$-space $H^3$, 
 $\Sigma_\#$ is the singular set of the map induced by 
 the normal vector field $\nu\colon{}M^2\to S^3_1$
 into the de Sitter $3$-space $S^3_1$.
\end{remark}

We write the volume forms as in \eqref{eq:volume-form} as
\begin{equation}\label{eq:2-volume-forms}
 \begin{array}{llll}
   d\hat{A}   &:=d\hat{A}_\phi=\lambda\, du\wedge dv,\qquad
  &d\hat{A}_\#&:=d\hat{A}_\psi=\lambda_\#\, du\wedge dv, \\
   dA         &:=dA_\phi=|\lambda|\, du\wedge dv,\qquad
  &dA_\#      &:=dA_\psi=|\lambda_\# |\, du\wedge dv,
 \end{array}
\end{equation}
and
\begin{align*}
   M^+   &:= M_{\phi}^+=
        \{p\in M^2\setminus \Sigma\,;\, d\hat{A}=dA \},\\
   M^-   &:= M_{\phi}^-=
        \{p\in M^2\setminus \Sigma\, ;\, d\hat{A}=-dA \}, \\
   M^+_\#& := M_{\psi}^+=
        \{p\in M^2\setminus \Sigma_\#\, ;\, d\hat{A}_\#=dA_\#\},\\
   M^-_\#& := M_{\psi}^-=
        \{p\in M^2\setminus \Sigma_\#\, ;\, d\hat{A}_\#=-dA_\# \}.
\end{align*}
Denote by $K$ (resp.\ $K_\#$) the Gaussian curvature 
of the metric $\first=ds^2_{\phi}$ 
(resp.\ $\third=ds^2_{\psi}$), which is defined on 
$M^2\setminus\Sigma$ (resp. $M^2\setminus\Sigma_\#$).

The extrinsic curvature $K^{\ext}$ in
Definition~\ref{def:ext} can be considered as a 
$C^{\infty}$-function on $M^2\setminus\Sigma$, because we are working
with the $2$-dimensional case.
Similarly, we denote by $K^{\ext}_{\#}$ the extrinsic
curvature with respect to the second homomorphism $\psi$,
namely, it is obtained by replacing $\first$ with $\third$ 
in \eqref{eq:ext}.
If we denote by $\widehat\first$, $\widehat\second$ and
$\widehat\third$ the representation matrices
of $\first$, $\second$ and $\third$, respectively,
with respect to $\{\partial/\partial u,\partial/\partial v\}$:
\begin{equation}\label{eq:fundamental-matrices}
  \begin{aligned}
  \widehat\first&:=
   \begin{pmatrix}
    \inner{\phi_u}{\phi_u} & \inner{\phi_u}{\phi_v}\\
    \inner{\phi_v}{\phi_u} & \inner{\phi_v}{\phi_v}
   \end{pmatrix},\\
  \widehat\second
   &:=
   -\begin{pmatrix}
    \inner{\phi_u}{\psi_u} & \inner{\phi_u}{\psi_v}\\
    \inner{\phi_v}{\psi_u} & \inner{\phi_v}{\psi_v}
   \end{pmatrix},\\
  \widehat\third
   &:=
   \begin{pmatrix}
    \inner{\psi_u}{\psi_u} & \inner{\psi_u}{\psi_v}\\
    \inner{\psi_v}{\psi_u} & \inner{\psi_v}{\psi_v}
   \end{pmatrix},
  \end{aligned}
\end{equation}
then the extrinsic curvatures are expressed as
\begin{equation}\label{eq:k-ext}
     K^{\ext}=\frac{\det\widehat\second}{\det\widehat\first},\qquad
     K_{\#}^{\ext}=\frac{\det\widehat\second}{\det\widehat\third}.
\end{equation}
Then \eqref{eq:ext2} is equivalent to 
\begin{equation}\label{eq:2-gauss}
  K=c+K^{\ext}.
\end{equation}

\begin{lemma}\label{lem:unbdd}
  Let $f:M^2\to N^3(c)$ be a front.
  Then
 \begingroup
 \renewcommand{\theenumi}{{\rm(\alph{enumi})}}
 \renewcommand{\labelenumi}{{\rm(\alph{enumi})}}
  \begin{enum}
   \item\label{item:k:1} 
     $K\,d\hat{A}=K_\#\,d\hat{A}_\#$,
   \item\label{item:k:1b} 
     $K_\#=1$ if $c=0$,
   \item\label{item:k:2} 
    $K^{\ext} \,d\hat{A}=d\hat{A}_\#$ and 
    $K^{\ext}_\# \,d\hat{A}_\#=d\hat{A}$,
   \item\label{item:k:2b} 
    $|K^{\ext}| \,dA=dA_\#$ and 
    $|K^{\ext}_\#| \,d{A}_\#=d{A}$,
   \item\label{item:k:3}
    $K^{\ext}\,K^{\ext}_\#=1$.
  \end{enum}
 \endgroup
\end{lemma}
\begin{proof}
 Take a positive
 (local) orthonormal frame field $\{\vect{e}_1,\vect{e}_2\}$
 of $\E_f$, and take a one-form $\omega$ as in
 \eqref{eq:conn-form},
 that is, $\omega$ is the connection form of $D$ with respect to 
 the frame $\{\vect{e}_1,\vect{e}_2\}$.
 By \eqref{eq:K},
 $K\,d\hat{A}$ is determined 
 by just the connection $D$.
 Hence we have \ref{item:k:1}.
 Take $2\times 2$-matrix-valued functions $G$ and $G_\#$ as
 \begin{equation}\label{eq:matrix-G}
     (\phi_u,\phi_v)=G(\vect{e}_1,\vect{e}_2),\qquad
     (\psi_u,\psi_v)=G_\#(\vect{e}_1,\vect{e}_2).
 \end{equation}
 By definition, $d\hat A$ and $d\hat A_\#$ in 
 \eqref{eq:2-volume-forms} are expressed as 
 \begin{equation}\label{eq:dA-G}
  \begin{aligned}
       d\hat A &= \lambda\,du\wedge dv=(\det G) \,du\wedge dv,\\
      d\hat A_\# &=\lambda_\#\,du\wedge dv= (\det G_\#) \,du\wedge dv.
  \end{aligned}
 \end{equation}
 In particular, we have that
 \begin{equation}\label{eq:dA-G0}
       dA =|\det G| \,du\wedge dv,\qquad
      dA_\# =|\det G_\#| \,du\wedge dv.
 \end{equation}
 On the other hand,
 the matrices $\widehat\first$, $\widehat\second$ and
 $\widehat\third$ as in \eqref{eq:fundamental-matrices}
 are written as
 \[
     \widehat\first = \trans{G}G,\qquad
     \widehat\second = -\trans{G}G_\#=-\trans{G_\#}G,\qquad
     \widehat\third = \trans{G_\#}G_\#,
 \]
 where $\trans(~)$ denotes  transposition.
 Thus, by \eqref{eq:k-ext}, we have
 \begin{equation}\label{eq:K-G}
      K^{\ext} = \frac{\det G_\#}{\det G\hphantom{_\#}},\quad
      K_\#^{\ext} = \frac{\det G\hphantom{_\#}}{\det G_\#}.
 \end{equation} 
 Hence \ref{item:k:2}, \ref{item:k:2b} and \ref{item:k:3} hold.
 Finally, if $c=0$, $\third$ is just the pull-back
 metric of canonical metric of the unit sphere by the
 Gauss map, so \ref{item:k:1b} follows.
\end{proof}

Now, we assume both of the singular sets $\Sigma$ and $\Sigma_\#$
consist of $A_2$- points and $A_3$-points.
Then applying two abstract formulas \eqref{eq:B} and
\eqref{eq:A} for $\phi$ and $\psi$ respectively,  
we have the following four Gauss-Bonnet formulas:
\begin{align}
 \int_{M^2} K\, d\hat{A}
 &=
   \int_{M^+}K\, dA-\int_{M^-}K\, dA  \label{eq:1p}\\
 &=
   2\pi\big(\chi(M^+)-\chi(M^-)\big)
            +2\pi \big(S^+-S^-\big), 
 \nonumber
 \\
 \int_{M^2} K\, dA
 &=2\pi\chi(M)-2\int_{\Sigma}\kappa\,d\tau,   
 \label{eq:1m}\\
 \int_{M^2} K_\#\, d\hat{A}_\#
 &=
    \int_{M_\#^+}K_\#\, dA_\#-\int_{M_\#^-}K_\#\, dA_\#
 \label{eq:2p}\\
 &= 2\pi\big(\chi(M_\#^+)-\chi(M_\#^-)\big)
        +2\pi \big(S_\#^+-S_\#^-\big), 
      \nonumber \\
  \int_{M^2} K_\#\, dA_\#&=2\pi\chi(M^2)-
       2\int_{\Sigma_\#}\kappa_\#\, d\tau_\#,
  \label{eq:2m}
\end{align}
where $\kappa$ (resp.\ $\kappa_\#$) is the singular curvature
function along $A_2$-points in $\Sigma$ (resp.\ $\Sigma_\#$) 
as in \eqref{eq:singular-curvature-2},
$d\tau$ (resp.\ $d\tau_\#$) is the length element
on the singular curve
with respect to $\first$ (resp.\ $\third$),
$S^+$ (resp.\ $S^-$) is the number of positive (resp.\ negative)
$A_3$-points of $\phi$,
and $S_\#^+$ (resp.\ $S_\#^-$) is the number of positive 
(resp.\ negative) $A_3$-points of $\psi$.
\begin{remark}\label{rem:peak2}
 As seen in Remark~\ref{rem:peaks}, 
 formulas \eqref{eq:1p}--\eqref{eq:2m} hold for fronts such that
 the singular sets of $\phi$ and $\psi$ consist of at most peak.
 Here, $S^+$ (resp.\ $S^-$) in \eqref{eq:1p} should 
 be considered as a number of positive  (resp.\ negative) 
 peaks of $\phi$, and
 $S^+_\#$ (resp.\ $S^-_\#$) in \eqref{eq:2p} should 
 be considered as a number of positive  (resp.\ negative) 
 peaks of $\psi$,
 in this case.
\end{remark}

\subsection{Applications of the four Gauss-Bonnet formulas}
Let $f:M^2\to N^3(c)$ be an immersion of a compact orientable $2$-manifold
$M^2$.
Then the second homomorphism $\psi_f$ is the shape operator of $f$,
and the set of singular points $\Sigma_\#$ of it
coincides with the 
{\it inflection points\/} of $f$
(see \cite{SUY4}).
Then $A_2$-points and $A_3$-points in $\Sigma_\#$ are called
{\it $A_2$-inflection points\/} and 
{\it $A_3$-inflection points}, respectively.
\begin{theorem}[A generalization of the Bleecker-Wilson formula]
\label{thm:a}
 Let $M^2$ be 
 a compact oriented $2$-manifold
 and $f:M^2\to N^3(c)$ an immersion.
 Suppose that the set of inflection points of $f$ consists of 
 $A_2$-points and $A_3$-points.
 Then we have
\begin{equation}\label{eq:BWnew}
       2\chi(M^-_\#)=S_\#^+-S_\#^-.
\end{equation}
 Moreover, $M^-_\#$ coincides with 
the set $\{p\in M^2\,;\,K^{\ext}(p)<0\}$.
\end{theorem}

\begin{proof}
 Since $A_2$-inflection points and $A_3$-inflection points are non-degenerate 
 singular points of $\psi$, the set of inflection points $\Sigma_\#$
 is a regular submanifold of a compact manifold $M^2$.
 Hence $\Sigma_\#$ is the disjoint union of a finite number of
 closed curves in $M^2$, 
 and thus the Euler number $\chi(\Sigma_\#)$ vanishes.
 Then we have
 \begin{equation}\label{eq:euler-plus-minus}
  \chi(M^2) = \chi(M_\#^+)+\chi(M_\#^-)+\chi(\Sigma_\#)
            = \chi(M_\#^+)+\chi(M_\#^-).
 \end{equation}
 Since $f$ is an immersion, $dA=d\hat A$ holds, and then
 \begin{alignat*}{2}
  \chi(M^2) &= \frac{1}{2\pi}\int_{M^2} K\,d\hat A \qquad 
            &&\text{(by the Gauss-Bonnet formula)}\\
            &= \frac{1}{2\pi}\int_{M^2} K_{\#}\,d\hat A_{\#}\qquad
            &&\text{(by \ref{item:k:1} in Lemma~\ref{lem:unbdd})}\\
            &= \chi(M_\#^+)-\chi(M_\#^-)+S_\#^+-S_\#^-\quad
            &&\text{(by \eqref{eq:2p})}\\
            &= \chi(M^2)-2\chi(M_\#^-)+S_\#^+-S_\#^-\quad
            &&\text{(by \eqref{eq:euler-plus-minus})}.
 \end{alignat*}
 Thus, we have the equality.
 Here, since $dA=d\hat A$, \ref{item:k:2} and \ref{item:k:2b}
 of Lemma~\ref{lem:unbdd}
 implies that $\Sigma_\#=\{p\in M^2\,;\,K^{\ext}=0\}$.
 Then 
 \begin{alignat*}{2}
    M_\#^- &=  \{p\in M^2\setminus\Sigma_\#\,;\,d\hat A_\#=-dA_\#\} \qquad
           &&  \\
           &=  \{p\in M^2\setminus\Sigma_\#\,;\,K^{\ext}\,d\hat A=-dA_\#\}\qquad
           &&  \text{(by \ref{item:k:2} in Lemma~\ref{lem:unbdd})}\\
           &=  \{p\in M^2\setminus\Sigma_\#\,;\,K^{\ext}\,dA=-dA_\#\} \qquad
           &&  \text{(since $d\hat A=dA$)}\\
           &=  \{p\in M^2\setminus\Sigma_\#\,;\,K^{\ext}<0\}
           &&  \text{(by \ref{item:k:2b} in Lemma~\ref{lem:unbdd})}\\
           &=  \{p\in M^2\,;\,K^{\ext}<0\}.
           &&
 \end{alignat*}
 Hence we have the conclusion.
\end{proof}

\begin{remark}\label{rmk:BW}
When $N^3(c)=\R^3$, \eqref{eq:BWnew} 
is the classical formula given in \cite{BW}.
In this case, $S_\#^++S_\#^-$ is equal to the total number of
cusps appears in the Gauss map $\nu$ of $f$.
Romero-Fuster \cite{R} discusses on this number for embedded
surfaces. 
When $N^3(c)=S^3$ or $H^3$, the formula has been also proved in the
authors' previous work \cite{SUY4}.
However, this formula is a generalization of 
those results in \cite{BW} and \cite{SUY4}, 
since $N^3(c)$ might not be simply connected in our setting.
It should be also remarked that the formula \eqref{eq:BWnew} 
can be generalized
to any compact immersed surfaces in
an arbitrary orientable Riemannian 3-manifold:
In fact, \eqref{eq:c} for $\psi_f$ is equivalent
to the Codazzi equation for the space form
and does not hold for a general Riemannian $3$-manifold.
However, \eqref{eq:2p} still
holds without assuming \eqref{eq:c},
as mentioned in Remark \ref{rem:GB}, and the above proof
works in this general setting (see \cite{SUY5} for details).
\end{remark}

Here, we give two examples satisfying $\chi(M^-)<0$.
\begin{example}
 It is well-known that there are embedded triply periodic minimal
 surfaces. 
 Although the Gauss maps of minimal surfaces only have isolated
 singular points, if we perturb the surface, 
 we get an immersion
 $f:M^2\to T^3$,
 where $M^2$ is a compact $2$-manifold with positive genus.
 Since minimal surfaces have negative Gaussian curvature,
 the perturbation
 $f$ satisfies $\chi(M^-)<0$, and the Gauss map of $f$
 must have cusps.  
\end{example}

\begin{example}\label{ex:trinoid}
 Similarly, considering the Jorge-Meeks symmetric trinoid 
 with the three ends rounding off to become closed discs, 
 one can get a closed surface
 in $\R^3$ as in Fig.~\ref{fig:trioid}.
 Then the resulting surface 
 satisfies $\chi(M^-)<0$, and its Gauss map 
 must have cusps.
\end{example}

\begin{figure}[htb]
 \begin{center}
  \vspace*{2mm}
        \includegraphics[width=3.3cm]{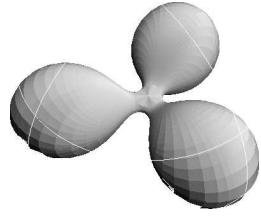}
  \vspace*{2mm}
  \caption{Trinoid with $\chi(M^-)<0$ in Example~\ref{ex:trinoid}.}
\label{fig:trioid}
\end{center}
\end{figure}

Interchanging the roles of $\phi$ and $\psi$,
we get the following dual assertion in $\R^3$.
Let $f\colon{}M^2\to \R^3$ be a front with unit normal vector
$\nu$.
Then for each real number $t$, 
$f_t:=f+t\nu$
is a front with unit normal $\nu$, which is called the 
{\em parallel front\/}
of signed distance $t$ of the front $f$.
\begin{theorem}[The dual version of Theorem~\ref{thm:a}]\label{thm:b}
 Let $f:S^2\to \R^3$ be a strictly convex surface and
 $f_t\,\,(t\in\R)$  a parallel front of $f$,
 Assume that the set of singular points of $f_t$ consists of
 cuspidal edges and swallowtails.
 Then
 $2\chi(M^-_{f_t})=S^+_{f_t}-S^-_{f_t}$
 holds,
 where $S^+_{f_t}$ {\rm(}resp. $S^-_{f_t}${\rm)} is the number of
 positive {\rm(}resp. negative{\rm)}
 swallowtails of $f_t$, and
 $M^{-}_{f_t}$ is the subset of 
 $M^2$ where the Gaussian curvature $K_t$ of $f_t$ is negative.
\end{theorem}

\begin{proof}
 Since $f$ is strictly convex, 
 $\nu\colon{}S^2\to S^2$ is a diffeomorphism,
 and then $\Sigma_\#$ vanishes.
 Then, exchanging the roles of $\phi$ and $\psi$, the proof of 
 Theorem~\ref{thm:a} implies the conclusion.
\end{proof}

\begin{example}\label{ex:ellipsoid}
 Let $f\colon S^2 \to\R^3$ be a parametrization of the
 ellipsoid defined by
 $(x^2/5)^2+(y/4)^2+z^2=1$
 and let $\nu$ be its unit normal vector field.
 Then we can observe that
 $f_c:=f+c\nu$ for $c=11/2$
 has four negative swallowtails
 and the Euler number of the set $\{K_c<0\}$
 is $-2$, where $K_c$ is the Gaussian curvature
 of $f_c$,
 see Fig.~\ref{fig:ellipsoid}.
\end{example}

\begin{figure}[htb]
 \begin{center}
        \includegraphics[height=4cm]{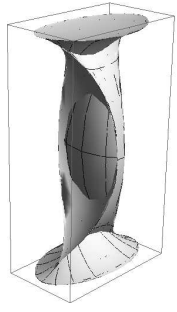}\hspace{1cm}
  \caption{A parallel surface of an ellipsoid with $\chi(M^-)<0$
         in Example~\ref{ex:ellipsoid}.}
  \label{fig:ellipsoid}
\end{center}
\end{figure}
On the other hand, as a conclusion of \eqref{eq:2m},
we have:
\begin{theorem}\label{thm:c}
 Let $f:M^2\to \R^3$ be an immersion and $\nu:M^2\to S^2$ the Gauss map
 of $f$.
 Assume that the singular set $\Sigma_\#$ of $\nu$ consists of 
 folds {\rm(}i.e. $A_2$-points{\rm)} and cusps 
 {\rm(}i.e. $A_3$-points{\rm)}.
 Then
 \[
    \int_{\Sigma_{\#}}\kappa_\#\,d\tau_{\#}
    =
      \int_{M_\#^-} K\,d\hat{A}
    =
      \int_{M^2}K^-\,dA
 \]
 holds, where $K^-=\min(0,K)$.
 In particular, the total dual singular curvature
 {\rm(}with respect to $\nu${\rm)}
 $\int_{\Sigma_{\#}}\kappa_\#\,d\tau_\#$ is non-positive.
\end{theorem}

\begin{proof}
 Since $f$ is an immersion into $\R^3$ (i.e. $c=0$), 
 $M^-_{\#}=\{p\in M^2\,;\,K^{\ext}<0\}=\{p\in M^2\,;\,K<0\}$
 holds.
 Then we have 
 \begin{equation}\label{eq:negint}
   \int_{M_\#^-} K\,dA = \int_{M^2} K^-\,dA.
 \end{equation}
 Here,
 \begin{alignat*}{2}
  \int_{M^2}K_{\#}\,dA_\#&=
     \int_{M_\#^+}K_\#\,d\hat A_\# -
     \int_{M_\#^-}K_\#\,d\hat A_\# 
     &&\\
   &=\int_{M_\#^+}K\,d\hat A -
     \int_{M_\#^-}K\,d\hat A \qquad
     &&\text{(by \ref{item:k:1} of Lemma~\ref{lem:unbdd})}\\
   &=\int_{M_\#^+}K\,dA -
     \int_{M_\#^-}K\,dA \qquad
     &&\text{(since $f$ is an immersion)}\\
   &=\int_{M^2}K\,dA - 2\int_{M_\#^-}K\,dA
     && \\
   &= 2\pi\chi(M^2) -2 \int_{M^2} K^-\,dA.
     && \text{(by \eqref{eq:negint})}.
 \end{alignat*}
 Thus, by \eqref{eq:2m}, we have the conclusion.
\end{proof}

\begin{example}\label{ex:sin}
 Consider a rotation of the sine curve
 \[
    f(u,v):=(u \cos v, u  \sin v, \cos u)
     \qquad (0< u<\pi,~0\le v\le 2\pi)
 \]
 whose Gauss map has $A_2$-singular points
 of positive singular curvature. 
 This shows that the singular
 curvature of Gauss maps  can take positive values. 
\end{example}

\begin{corollary}\label{cor:c1}
 Under the same assumptions as in
 Theorem~\ref{thm:c},
 if there exists a point $p$ satisfying
 $K(p)<0$ then there exists a fold point such that
 $\kappa_\#$ is negative.
\end{corollary}

\begin{example}
 An embedded Delaunay surface (an unduloid)
 as a rotationally symmetric periodic surface
 in  $\R^3$ can be considered as an immersion
 into a flat space
 $f:S^1\times S^1 \to \R^2\times S^1$
 which has an annular domain having 
 negative Gaussian curvature.
 As a consequence, the image of its Gauss map
 of $f$ has folds of negative singular curvature.
\end{example}

\begin{remark}\label{cor:c2}
 Let $f:M^2\to \R^3$ be an immersion of 
 a compact 2-manifold $M^2$.
 Then the inequality
 \[
   \frac{1}{2\pi}\int_{M^2}|K|\, dA\ge \beta_0+\beta_1+\beta_2
 \]
 is called the Chern-Lashof inequality, where
 $\beta_j$ $(j=0,1,2)$ is the $j$-th Betti number of $M^2$.
 The equality holds if and only if $f$ is tightly immersed in $\R^3$. 
 On the other hand, by the Gauss-Bonnet formula,
 it holds that
 \[
     \frac{1}{2\pi}\int_{M^2}K\, dA = \beta_0-\beta_1+\beta_2.
 \]
 Hence, we have that
 \[
    -\int_{\Sigma_\#}\kappa_\#\,d\tau_\#=
    -\int_{M^2}K^-\,dA\ge  2\pi\beta_1,
 \]
 where the equality holds if and only if $f$ is a tight 
 immersion. 
\end{remark}

Similar to Theorem~\ref{thm:b}, the dual version of Theorem~\ref{thm:c}
holds:
\begin{theorem}\label{thm:d}
 Let $f:S^2\to \R^3$ be a strictly convex immersion and
 $f_t$ a parallel front of $f$  for $t\in\R$.
 Assume that the set of singular points $\Sigma_{f_t}$
 of $f_t$ consists of  cuspidal edges and swallowtails.
 Then
 \[
    \int_{\Sigma_{f_t}}\kappa_{f_t}\,d\tau_{f_t}
        =
    \int_{M^{-}_{f_t}}K_\#\,d\hat{A}_\#
    =
    \int_{M^{-}_{f_t}}K_{t}\,d\hat{A}_{f_t}
 \]
 holds,
 where 
 $K_t$ is the Gaussian curvature of $f_t$,
 $M^-_{f_t}\subset M^2$ is the set where $K_t<0$,
 and  $\kappa_{f_t}$ is the singular curvature of cuspidal edges
 of $f_t$.
\end{theorem}

\begin{proof}
 The equality
 $\int_{\Sigma_{f_t}}\kappa_{f_t}\,d\tau_{f_t}
        =
    \int_{M^{-}_{f_t}}K_\#\,d\hat{A}_\#$ 
 follows from the proof of Theorem \ref{thm:c} by exchanging the
 roles of $\phi$ and $\psi$.
 On the other hand, in the proof of Theorem~\ref{thm:a},
 we proved that
 $M^{-}_{\#}$ coincides with the set $\{p\in M^2\,;\,K^{\ext}(p)<0\}$
 because of $\phi_f$ has no singular points.
 In our situation, since 
 $\nu\colon{}S^2\to S^2$ is a diffeomorphism,
 $\psi_f$ has no singular points.
 Thus, by exchanging the roles of $\phi$ and $\psi$,
 we have that
 \[
   M^{-}_{f_t}=\{p\in M^2\,;\,(K^{\ext}_t)_\#(p)<0\}
     =\{p\in M^2\,;\,K^{\ext}_t(p)<0\},
 \]
 because of \ref{item:k:3} in Lemma \ref{lem:unbdd}.
 Moreover, by \eqref{eq:2-gauss}, we have $K^{\ext}_t=K_t$
 and get the assertion.
\end{proof}
\begin{corollary}\label{cor:d1}
 Under the same assumptions as in Theorem \ref{thm:d},
 if there exists a point $p$ satisfying
 $K_{f_t}(p)<0$ then there exists a point such that
 $\kappa_{f_t}$ is negative.
\end{corollary}

\subsection{The case of bounded Gaussian curvature}
From now on, we assume that
the Gaussian curvature is bounded.
As shown below, there are many such surfaces as
wave fronts.
The following lemma holds.

\begin{lemma}\label{lem:bdd}
 Let $M^2$ be an oriented $2$-manifold,
 $f:M^2\to N^3(c)$ a front, and $p$ an $A_2$-singular point of $f$.
 Let
 $(M^2,\E_f,\inner{~}{~},D,\phi,\psi)$ 
 be a front bundle 
 associated to $f$
 {\rm(}cf.\  Example~\ref{ex:front-bundle}\/{\rm)}.
 Suppose that there exists a neighborhood 
 $U$ of a $(\phi$-$)$singular point $p$
 such that
 $\log |K^{\ext}|$ is bounded on $U\setminus\Sigma$, where $K^{\ext}$
 is the
 extrinsic  Gaussian curvature and $\Sigma$ is the set of singular
 points of $f$.

 Then the following holds on $U${\rm:}
 \begingroup
 \renewcommand{\theenumi}{{\rm(\alph{enumi})}}
 \renewcommand{\labelenumi}{{\rm(\alph{enumi})}}
 \begin{enum}
 \setcounter{enumi}{4}
 \item\label{item:k:4} 
       $\Sigma=\Sigma_\#$.
 \item\label{item:k:5} 
       By setting 
       $\epsilon:=\sgn(K^{\ext}|_U)$, it holds that
       $M^+=M_\#^\epsilon$ and
       $\kappa\,d\tau=\epsilon\kappa_\#\,d\tau_\#$,
       where $K^{\ext}|_U$ is the restriction of the function
       $K^{\ext}$ on $U$.
 \end{enum}
 \endgroup
\end{lemma}

\begin{proof}
 We fix a positive orthonormal frame field $\{\vect{e}_1,\vect{e}_2\}$
 of $\E_f$ and take matrices $G$ and $G_\#$ as in \eqref{eq:matrix-G}.
 Since $K^{\ext}$ is bounded on $U$,
 $\det G_\#=0$ holds on $\Sigma\cap U=\{\det G=0\}$ because of
 \eqref{eq:K-G}.
 This implies $\Sigma\subset \Sigma_\#$.

 On the other hand, by \ref{item:k:3} of Lemma~\ref{lem:unbdd}
 and the assumptions, $K_\#^{\ext}=1/K^{\ext}$ is also bounded.
 Then exchanging the roles of $\phi$ and $\psi$, 
 we have $\Sigma\supset
 \Sigma_\#$.
 Hence \ref{item:k:4} holds.

 We shall prove \ref{item:k:5}.
 Parametrize the singular set $\Sigma=\Sigma_\#$
 by a regular curve $\gamma(t)$ on $U$.
 Since $p$ is an $A_2$-point, 
 $\phi\bigl(\dot\gamma(t)\bigr)$ does not vanish,
 and one can take a $C^{\infty}$-function $\theta=\theta(t)$
 such that
 \begin{equation}\label{eq:phi-dot}
     \phi\bigl(\dot\gamma(t)\bigr)
     = \left|\phi\bigl(\dot\gamma(t)\bigr)\right|
       \bigl(
          \cos\theta(t)\vect{e}_1 +\sin \theta(t)\vect{e}_2
       \bigr)
 \end{equation}
 where $\dot{~}=d/dt$ and $\vect{e}_j=\vect{e}_j\bigl(\gamma(t)\bigr)$
 ($j=1,2$).
 Then the conormal vector $\vect{n}(t)$ along $\gamma(t)$
 is expressed as 
 \[
     \vect{n}(t)=
     -\sin\theta(t)\vect{e}_1 + \cos\theta(t) \vect{e}_2.
 \]
 Let $\omega$ be the connection form of $D$ with respect 
 to the frame $\{\vect{e}_1,\vect{e}_2\}$
as in \eqref{eq:conn-form}.
Then we have
 \[
     D_{d/dt}\vect{n}(t) = 
    \left(\omega\bigl(\dot\gamma(t)\bigr)-\dot\theta(t)\right)
    \frac{\phi\bigl(\dot\gamma(t)\bigr)}{%
          \left|\phi\bigl(\dot\gamma(t)\bigr)\right|}.
 \]
 Substituting this into \eqref{eq:singular-curvature-2},
 we have
 \[
     \kappa(t)=
     -\frac{\sgn\left(d\lambda\bigl(\eta(t)\bigr)\right)}{%
          \left|\phi\bigl(\dot\gamma(t)\bigr)\right|}
    \left(\omega\bigl(\dot\gamma(t)\bigr)-\dot\theta(t)\right),
 \]
 where $\eta(t)$ is a null vector field such that
 $\{\dot\gamma(t),\eta(t)\}$ is positively oriented.
 Then we have
 \begin{equation}\label{eq:bounded-curv}
    \kappa(t)\,d\tau =
    \kappa(t)\,\left|\phi\bigl(\dot\gamma(t)\bigr)\right|\,dt
    =  -\sgn\left(d\lambda\bigl(\eta(t)\bigr)\right)
    \left(\omega\bigl(\dot\gamma(t)\bigr)-\dot\theta(t)\right)\,dt.
 \end{equation}
 Similarly, if we set
 \begin{equation}\label{eq:psi-dot}
     \psi\bigl(\dot\gamma(t)\bigr)
     = \left|\psi\bigl(\dot\gamma(t)\bigr)\right|
       \bigl(
          \cos\theta_\#(t)\vect{e}_1 +\sin \theta_\#(t)\vect{e}_2
       \bigr),
 \end{equation}
 we have
 \begin{equation}\label{eq:bounded-curv-dual}
     \kappa_{\#}(t)\,d\tau_\#=
     -\sgn\left(d\lambda_\#\bigl(\eta_\#(t)\bigr)\right)
    \left(\omega\bigl(\dot\gamma(t)\bigr)-\dot\theta_\#(t)\right)\,dt,
 \end{equation}
 where $\eta_\#(t)$ is a null vector field with respect to $\psi$
 such that $\{\dot\gamma(t),\eta_\#(t)\}$ is positively oriented.

 Both of the frame fields $\{\dot\gamma(t),\eta(t)\}$ and 
 $\{\dot\gamma(t),\eta_\#(t)\}$ are 
 compatible to the orientation of $M^2$.
 Since  $\lambda=0$ on $\gamma(t)$,
 the relations \eqref{eq:dA-G} and \eqref{eq:K-G} yield 
 \[
      \sgn d\lambda\bigl(\eta(t)\bigr) =
      \sgn d\lambda\bigl(\eta_\#(t)\bigr)
      =
      \left(\sgn(K^{\ext}|_U)\right)
      \left(\sgn d\lambda_\#\bigl(\eta_\#(t)\bigr)\right).
 \]
 Since $\epsilon=\sgn(K^{\ext}|_U)$, we have
 \begin{equation}\label{eq:sign}
      \sgn d\lambda\bigl(\eta(t)\bigr) =
      \epsilon\, d\lambda_\#\bigl(\eta_\#(t)\bigr).
 \end{equation}
 Finally, by \cite[Theorem 3.1]{SUY1}, 
 the second fundamental form $\second$ vanishes on $\gamma(t)$.
 Thus by \eqref{eq:phi-dot} and \eqref{eq:psi-dot},
 we have
 \begin{align*}
    0 &= -\second\bigl(\dot\gamma(t),\dot\gamma(t)) 
       = -\inner{\phi\bigl(\dot\gamma(t)\bigr)}{%
            \psi\bigl(\dot\gamma(t)\bigr)}\\
      &= -\left|\phi(\dot\gamma)\right|
         \left|\psi(\dot\gamma)\right| 
        \cos\bigl(\theta(t)-\theta_\#(t)\bigr),
 \end{align*}
 and thus $\theta(t)-\theta_{\#}(t)$ is constant.
 Hence 
 \begin{equation}\label{eq:dot-theta}
    \dot\theta(t) = \dot\theta_{\#}(t).
 \end{equation}
 Summing up, by \eqref{eq:bounded-curv},
 \eqref{eq:bounded-curv-dual}, \eqref{eq:sign}
 and \eqref{eq:dot-theta}, we have
 \ref{item:k:5}.
\end{proof}

\begin{corollary}\label{thm:add}
 Let $M^2$ be a compact oriented $2$-manifold, and 
 $f:M^2\to \R^3$ a front such that $K>0$ and
 $\log|K|$ is bounded
 on $M^2\setminus \Sigma$.
 Then the singular curvature at the $A_2$-points 
of the Gauss map of $f$ is always negative.
\end{corollary}

\begin{example}\label{ex:p}
 We set
 \[
    f(u,v):=\bigl( (2-\cos u)\cos v,(2-\cos u)\sin v ,u-\sin u\bigr).
 \]
 which is a rotation of a cycloid.
 By a straightforward calculation, $K$ is bounded.
 In particular, the Gauss map of $f$ has a fold of
 negative singular curvature.
\end{example}

\begin{theorem}\label{thm:e}
 Let $M^2$ be a compact oriented $2$-manifold, and 
 $f:M^2\to N^3(c)$ a front such that $\log|K^{\ext}|$ is bounded
 on $M^2\setminus\Sigma$.
 Suppose that
 the singular sets $\Sigma$ 
 of $\phi$ and $\Sigma_\#$ ($=\Sigma$) of $\psi$
 consist of $A_2$-points and $A_3$-points.
 Then
 \[
    S^+ - S^- = \sgn(K^{\ext})(S_\#^+-S_\#^-)
 \]
 holds,
 where $S^+$ {\rm(}resp.\ $S^-${\rm)} is the number of positive 
 {\rm(}resp.\ negative{\rm)}
 swallowtails of $f$,
 and $S_\#^+$ {\rm(}resp.\ $S_\#^+${\rm)} is the number of positive
 {\rm(}resp.\ negative{\rm)} $A_3$-point of the map $\nu$ induced from
 the unit normal vector field of  $f$ as in Remark \ref{rmk:2}.
\end{theorem}

\begin{proof}
 Combining \ref{item:k:1} of Lemma~\ref{lem:unbdd}, \eqref{eq:1p},
 \eqref{eq:2p} and \ref{item:k:5} of Lemma~\ref{lem:bdd},
 we have the conclusion.
\end{proof}

\begin{example}\label{ex:CMC}
 Let $f\colon{}M^2\to \R^3$ be an immersion of constant mean curvature
 $1/2$ and $\nu:M^2\to S^2(\subset \R^3)$ its Gauss map.
 Then the parallel surface 
 $f_1=f +\nu$
 is a wave front
 with constant Gaussian curvature $1$.
 Since the unit normal vector field  of $f_1$ is also $\nu$,
 the singular set $\Sigma_\#$ of $\nu$ coincides with the set
 $\{K_f=0\}$, where  $K_f$ is the Gaussian curvature of $f$.
 Then by Lemma~\ref{lem:bdd}, $\Sigma=\Sigma_\#=\{K_f=0\}$.

 It is well-known that there are many constant mean curvature
 immersed tori in $\R^3$.
 On the other hand, 
 Gro\ss{}e-Brauckmann \cite{GB} constructed triply periodic 
 surfaces in $\R^3$ with constant mean curvature $1/2$.
 Such surfaces are compact surfaces in a $3$-dimensional
 flat torus, and then Theorem~\ref{thm:e} can be applied
 for $f +\nu$ of them.
\end{example}

Moreover, for fronts of negative extrinsic curvature,
we have the following:

\begin{theorem}\label{thm:g}
 Let $M^2$ be a compact oriented $2$-manifold, and 
 $f:M^2\to N^3(c)$ a front such that $K^{\ext}<0$ and
 $\log|K^{\ext}|$ is bounded on $M^2$.
 Then  $\chi(M^2)=0$ holds.
\end{theorem}

\begin{proof}
 By  Lemma~\ref{lem:bdd},
 we have $\Sigma=\Sigma_\#$ and 
 $\kappa\, d\tau=-\kappa_{\#}\,d\tau_\#$.
 Then by \eqref{eq:1m} and \eqref{eq:2m}, we have
 \begin{align*}
  \int_{M^2}\,K\,dA &=
      2\pi \chi(M^2) - 2 \int_{\Sigma}\kappa\,d\tau 
    = 2\pi\chi(M^2) + 2\int_{\Sigma}\kappa_{\#}\,d\tau_\# \\
    &= 2\pi\chi(M^2) + 
     \left(2\pi\chi(M^2)-
           \int_{M^2}K_{\#}\,dA_{\#}\right).
 \end{align*}
 On the other hand, by \ref{item:k:5} in Lemma~\ref{lem:bdd}
 and \ref{item:k:1} in Lemma~\ref{lem:unbdd}, we have
 \begin{align*}
  \int_{M^2}\!\!K_{\#}\,dA_\# &=
     \int_{M_\#^+}\!\!K_{\#}\,d\hat A_{\#} -
     \int_{M_\#^-}\!\!K_{\#}\,d\hat A_{\#} 
   =\int_{M^-}\!\!K_{\#}\,d\hat A_{\#} -
      \int_{M^+}\!\!K_{\#}\,d\hat A_{\#} \\
   &=  \int_{M^-}\!\!K\,d\hat A -
     \int_{M^+}\!\!K\,d\hat A = -\int_{M^2}\!\!K\,dA.
 \end{align*}
 Combining these, we have the conclusion.
\end{proof}

\begin{example}\label{ex:cyc}
 It is well-known that there is a front of constant Gaussian
 curvature $-1$ which is diffeomorphic to a torus, that is,
 Euler number vanishes.
 Moreover, we can construct an example of a compact front
 of non-constant negative Gaussian curvature as follows: 
 We set
 \[
    f(u,v):=\bigl((2+\cos u)\cos v ,(2+\cos u)\sin v ,u-\sin u\bigr),
 \]
 which is another rotation of a cycloid 
 (cf.\ Example \ref{ex:p}).
 One can easily check that
 $K$ is bounded.
 Since $f$ is periodic, it can be considered as a 
torus  in a flat space form $\R^2\times S^1$. 
\end{example}

\begin{example}
 Let $f\colon{}M^2\to S^3$ be a flat front in the unit $3$-sphere, 
 that is, $f$ has vanishing Gaussian curvature on its regular points.
 Then the unit normal vector field $\nu:M^2\to S^3$
 also gives a flat fronts.
 If $M^2$ is compact, $f$ is weakly complete in the sense of 
 \cite{KU} and $M^2$ is orientable (see \cite[Section 1]{KU}).
 On the other hand
 by Theorem~\ref{thm:g}, the Euler characteristic of $M^2$
 vanishes because $K^{\ext}=-1$ for flat surfaces in $S^3$.
 Thus $M^2$ is diffeomorphic to a torus.
 Moreover  the singular set of $f$
 coincides with that of $\nu$, 
 and  Theorem~\ref{thm:e} yields that
 $S^+-S^- = -(S^+_\#-S^{-}_\#)$
 holds. This identity is not trivial
 since
 the set of swallow tails of $f$
 is not equal to that of  $\nu$.
 (In fact, $p$ is an $A_3$-point, the null direction 
 should be tangential direction of the singular set $\Sigma$.
 On the other hand, null directions of $\phi$ and $\psi$ 
 are linearly independent because $f$ is a front.)
 Several other properties of flat tori in $S^3$ as fronts are
 discussed in \cite{KU}. 
\end{example}
\section*{Acknowledgement}
 The authors thank Wayne Rossman
 for careful reading of the first draft for giving valuable comments.
 They also thank the referee for valuable comments.

\end{document}